\documentclass[reqno,a4paper]{amsart}

\usepackage{enumitem}
\setenumerate{label=\textnormal{(\arabic*)}}

\usepackage{amsmath,amssymb,dsfont,verbatim,bm,array,mathtools,bm,geometry,fge}
\usepackage[latin1]{inputenc}
\usepackage{amsrefs}
\usepackage{booktabs}
\usepackage[raggedright]{titlesec}
\usepackage{mathtools}

\titleformat{\chapter}[display]
{\normalfont\huge\bfseries}{\chaptertitlename\\thechapter}{20pt}{\Huge}
\titleformat{\section}
{\normalfont\Large\bfseries\center}{\thesection}{1em}{}
\titleformat{\subsection}
{\normalfont\large\bfseries}{\thesubsection}{1em}{}
\titleformat{\subsubsection}[runin]
{\normalfont\normalsize\bfseries}{\thesubsubsection}{1em}{}
\titleformat{\paragraph}[runin]
{\normalfont\normalsize\bfseries}{\theparagraph}{1em}{}
\titleformat{\subparagraph}[runin]
{\normalfont\normalsize\bfseries}{\thesubparagraph}{1em}{}
\titlespacing*{\chapter} {0pt}{50pt}{40pt}
\titlespacing*{\section} {0pt}{3.5ex plus 1ex minus .2ex}{2.3ex plus .2ex}
\titlespacing*{\subsection} {0pt}{3.25ex plus 1ex minus .2ex}{1.5ex plus .2ex}
\titlespacing*{\subsubsection}{0pt}{3.25ex plus 1ex minus .2ex}{1.5ex plus .2ex}
\titlespacing*{\paragraph} {0pt}{3.25ex plus 1ex minus .2ex}{1em}
\titlespacing*{\subparagraph} {\parindent}{3.25ex plus 1ex minus .2ex}{1em}


\usepackage[raggedright]{titlesec}

\titleformat{\chapter}[display]
{\normalfont\huge\bfseries}{\chaptertitlename\\thechapter}{20pt}{\Huge}
\titleformat{\section}
{\normalfont\Large\bfseries\center}{\thesection}{1em}{}
\titleformat{\subsection}
{\normalfont\large\bfseries}{\thesubsection}{1em}{}
\titleformat{\subsubsection}
{\normalfont\normalsize\bfseries}{\thesubsubsection}{1em}{}
\titleformat{\paragraph}[runin]
{\normalfont\normalsize\bfseries}{\theparagraph}{1em}{}
\titleformat{\subparagraph}[runin]
{\normalfont\normalsize\bfseries}{\thesubparagraph}{1em}{}

\titlespacing*{\chapter} {0pt}{50pt}{40pt}
\titlespacing*{\section} {0pt}{3.5ex plus 1ex minus .2ex}{2.3ex plus .2ex}
\titlespacing*{\subsection} {0pt}{3.25ex plus 1ex minus .2ex}{1.5ex plus .2ex}
\titlespacing*{\subsubsection}{0pt}{3.25ex plus 1ex minus .2ex}{1.5ex plus .2ex}
\titlespacing*{\paragraph} {0pt}{3.25ex plus 1ex minus .2ex}{1em}
\titlespacing*{\subparagraph} {\parindent}{3.25ex plus 1ex minus .2ex}{1em}

\usepackage{titletoc}

\contentsmargin{2.55em}
\dottedcontents{section}[1.5em]{}{2em}{1pc}
\dottedcontents{subsection}[4.35em]{}{2.8em}{1pc}
\dottedcontents{subsubsection}[7.6em]{}{3.2em}{1pc}
\dottedcontents{paragraph}[10.3em]{}{3.2em}{1pc}

\newtheorem{theorem}{Theorem}[section]
\newtheorem{lemma}[theorem]{Lemma}
\newtheorem{proposition}[theorem]{Proposition}
\newtheorem{corollary}[theorem]{Corollary}

\theoremstyle{definition}
\newtheorem{definition}[theorem]{Definition}

\theoremstyle{remark}
\newtheorem{remark}[theorem]{Remark}

\DeclareMathOperator{\Aut}{Aut}

\DeclareMathOperator{\Supp}{Supp}

\DeclareMathOperator{\wdeg}{wdeg}

\newcommand{\ov}{\overline}

\newcommand{\wt}{\widetilde}

\newcommand{\BigFig}[1]{\parbox{12pt}{\Huge #1}}
\newcommand{\BigZero}{\BigFig{0}}

\begin{document}

\title{A system of polynomial equations related to the Jacobian Conjecture}

\author{Jorge A. Guccione}
\address{Departamento de Matem\'atica\\ Facultad de Ciencias Exactas y Naturales-UBA,
Pabell\'on~1-Ciudad Universitaria\\ Intendente Guiraldes 2160 (C1428EGA) Buenos Aires, Argentina.}
\address{Instituto de Investigaciones Matem\'aticas ``Luis A. Santal\'o"\\ Facultad de
Ciencias Exactas y Natu\-ra\-les-UBA, Pabell\'on~1-Ciudad Universitaria\\ Intendente
Guiraldes 2160 (C1428EGA) Buenos Aires, Argentina.}
\email{vander@dm.uba.ar}

\author{Juan J. Guccione}
\address{Departamento de Matem\'atica\\ Facultad de Ciencias Exactas y Naturales-UBA\\
Pabell\'on~1-Ciudad Universitaria\\ Intendente Guiraldes 2160 (C1428EGA) Buenos Aires, Argentina.}
\address{Instituto Argentino de Matem\'atica-CONICET\\ Savedra 15 3er piso\\ (C1083ACA)
Buenos Aires, Argentina.}
\email{jjgucci@dm.uba.ar}
\thanks{Jorge A. Guccione and Juan J. Guccione were supported by CONICET PIP 2021-2023 GI,11220200100423CO}

\author[C. Valqui]{Christian Valqui}
\address{Pontificia Universidad Cat\'olica del Per\'u, Secci\'on Matem\'aticas, PUCP,
Av. Universitaria 1801, San Miguel, Lima 32, Per\'u.}
\address{Instituto de Matem\'atica y Ciencias Afines (IMCA) Calle Los Bi\'ologos 245.
Urb San C\'esar. La Molina, Lima 12, Per\'u.}
\email{cvalqui@pucp.edu.pe}
\thanks{Christian Valqui was supported by PUCP-CAP 2023-PI0991}


\begin{abstract} We prove that the Jacobian conjecture is false if and only if there exists
a solution to a certain system of polynomial equations. We analyse the solution set of this
system. In particular we prove that it is zero dimensional.
\end{abstract}


\maketitle

\section*{Introduction}
Let $K$ be a characteristic zero field. The Jacobian Conjecture (JC) in dimension two, stated
by Keller in \cite{K}, says that any pair of polynomials $P,Q\in R:= K[x,y]$ with
$$
[P,Q]:=\partial_x P \partial_y Q - \partial_x Q \partial_y P\in K^{\times}
$$
defines an automorphism of $R$.

T. T. Moh analyses in \cite{M} the existence of possible counterexamples $(P,Q)$ with total
degree of $P$ and $Q$ lower than 101 and finds four exceptional cases $(m,n)=(48,64)$,
$(m,n)=(50,75)$, $(m,n)=(56,84)$ or $(m,n)=(66,99)$, where $(n,m)=(\deg(P),\deg(Q))$. Then
he discards these cases by hand solving certain Ad Hoc systems of equations for the
coefficients of the possible counterexamples. Motivated by this we introduce and begin the
study of a polynomial system~$S_t(n,m,(\lambda_i),F_{1-n})$ of $m+n-2$ equations with
coefficients in a commutative $K$-algebra $D$ and $m+n-2$ variables. Here
$(\lambda_i)_{0\le i\le m+n-2}$ is a family of $m+n-2$ elements of $K$ and $F_{1-n}\in D$.
Among other results, we prove that a particular instance of this system (with $D=K[y]$ and
$F_{1-n}=y$) has a solution in $D^{m+n-2}$ if and only if there exists a counterexample
$(P,Q)$ to JC  with $(n,m)=(\deg(P),\deg(Q))$. For this we use an equivalent formulation
of the JC due to Abhyankar \cite{A}, which asserts that JC is true if for all Jacobian
pairs $(P,Q)$ either $\deg(P)$ divides $\deg(Q)$ or viceversa.  We also prove that if
$D$ is an integral domain, then the set of solutions of $S_t(n,m,(\lambda_i),F_{1-n})$
is finite. After that, we analyse the case in which $\lambda_i=0$ for $i>0$, which we call
the homogeneous system, giving a very detailed description of its solutions. In
Proposition~\ref{zannier} we show that the homogeneous system has always a solution,
using a result of~\cite{Z} (See also~\cite{APZ}).

 Our system
provides a significative reduction of the number of equations and variables needed in order
to verify the existence of a counterexample to JC at $(n,m)$, where the most naive approach
needs $m(m+1)/2+n(n+1)/2$ variables and $(m+n-1)(m+n-2)/2$ equations. However the number of
equations is still too big to have a realistic chance to verify the existence of a
counterexample to JC for the pairs $(m,n)=(48,64)$, $(m,n)=(50,75)$, $(m,n)=(56,84)$ or $(m,n)=(66,99)$,
which are the cases found in~\cite{M}.

In the last section we show how one has to proceed in a concrete example, analysing the case
$(n,m)=(50,75)$. Using a reduction of degree technique as in Section~8 of \cite{GGV1}, one
can show that in that case there must exist a pair $(P,Q)\in K[x,y,y^{-1}]$ with
$(\deg_x(P),\deg_x(Q))=(4,6)$ or $(\deg_x(P),\deg_x(Q))=(6,9)$, satisfying certain additional
properties. Among others, the Jacobian $[P,Q]\notin K^{\times}$. Due to this fact we must use
a slight variation of the system~$S_t(n,m,(\lambda_i),F_{1-n})$. Our computations provide an
independent verification of Moh's result: There is no counterexample at $(50,75)$. An advantage
of our system of equations compared to the ones used by Moh, is its form, which is canonical
even for the modified systems. On one hand this allows to program more general algorithms in
order to verify concrete cases, following the procedure suggested in Section~5. On the other
hand, further analysis of the structure of the system of equations could give some progress
in solving the JC, discarding at least some infinite families of possible counterexamples,
and not only single cases.

\section{The Jacobian Conjecture as a system of equations}
Let $K$ be a characteristic zero field and let $D$ an arbitrary commutative $K$-algebra. In
this section we introduce a polynomial system $S_t(n,m,(\lambda_i),F_{1-n})$ of $m+n-2$
equations with $m+n-2$ variables, where $(\lambda_i)_{0\le i\le m+n-2}$ is a family of
$m+n-2$ elements of $K$ and $F_{1-n}\in D$. The main results are Theorem~\ref{principal}
and Corollary~\ref{caracterizacion}, in which we show that there exists a counterexample
$(P,Q)$ to JC with $(\deg(P),\deg(Q))=(m,n)$ if and only if  $S_t(n,m,(\lambda_i),y)$ has
a solution in $K[y]^{m+n-2}$ for some~$\lambda_1,\dots, \lambda_{m+n-2}\in K$.

A non-zero element $\mathbf{w}:=(w_1,w_2)\in \mathds{Z}^2$ is  called a {\em a direction}
if $\gcd(w_1,w_2) = 1$ and $w_1>0$ or $w_2>0$. In the sequel for each direction
$\mathbf{w}:=(w_1,w_2)$, we write $|\mathbf{w}|:=w_1+w_2$. Furthermore, by the sake of
simplicity we set $R:= K[x,y]$. A polynomial $P\in R$ is said to have a {\em Jacobian mate}
$Q\in R$ if
$$
[P,Q] := \partial_x P \partial_y Q -  \partial_y P \partial_x Q \in K^{\times}.
$$
In this case $P$ and $Q$ are called {\em  Jacobian polynomials} and $(P,Q)$ is called a
{\em Jacobian pair}.

\smallskip

To each direction $\mathbf{w}$ we associate the so-called $\mathbf{w}$-grading on $R$,
$$
R:= \bigoplus_{d\in \mathds{Z}} R_d(\mathbf{w}),
$$
where $R_d(\mathbf{w})$ is the the $K$-vector subspace of $R$ generated by all monomials
$x^iy^j$ such that $i w_1 + j w_2 = d$. If there is no confusion possible, we will write
$R_d$ instead of $R_d(\mathbf{w})$. For $P \in R\setminus\{0\}$ we denote by $P_+$ the
$\mathbf{w}$-homogeneous part of $P$ of highest degree. Furthermore, if
$P_+\in R_d(\mathbf{w})$, then we say that the $\mathbf{w}$-degree of $P$ is $d$, and write
$\mathbf{w}\!\deg(P) = d$. For convenience we set $\mathbf{w}\!\deg(0) = -\infty$. As usual
we will write $\deg(P)$, $\deg_x(P)$ and $\deg_y(P)$ instead of $(1,1)\!\deg(P)$,
$(1,0)\!\deg(P)$ and $(0,1)\!\deg(P)$, respectively. We also say that $P$ is homogeneous
if it is $(1,1)$-homogeneous. We have the
following result due to Abhyankar:

\begin{proposition}[\cite{vdE}*{Theorem 10.2.23}]\label{contraejemplo subrectangular}
The Jacobian conjecture is false if and only if there exists a Jacobian pair $(P,Q)$,
such that neither $\deg(P)$ divides $\deg(Q)$ nor $\deg(Q)$ divides $\deg(P)$.
\end{proposition}

\begin{remark}\label{re 1.2} The arguments in the proof of the above proposition show that
if $(P,Q)$ is a Jacobian pair such that neither $\deg(P)$ divides $\deg(Q)$ nor $\deg(Q)$
divides $\deg(P)$, then $(P,Q)$ is a counterexample to JC.
\end{remark}

We will use freely that if $\varphi$ is an automorphism of $R$, then
$$
[\varphi(P),\varphi(Q)] = \varphi([P,Q])[\varphi(x),\varphi(y)].
$$
Let $(P,Q)$ be as in Proposition~\ref{contraejemplo subrectangular}. For each
$\lambda\in K$ we define $\varphi_{\lambda}\in\Aut(R)$ by
$$
\varphi_{\lambda}(x):=x\quad\text{and}\quad \varphi_{\lambda}(y):=y+\lambda x.
$$
Let $n:=\deg(P)$, $m:=\deg(Q)$ and  $\mathbf{w}:=(1,0)$ . It is easy to check that there
exists $\lambda\in K$ such that $\varphi_{\lambda}(P)_+ = \mu_Px^n$ and
$\varphi_{\lambda}(Q)_+ = \mu_Q x^m$, with $\mu_P,\mu_Q\in K^{\times}$. Consequently,
since $\varphi_{\lambda}$ is $(1,1)$-homogeneous,
$$
\varphi_{\lambda}(P) = \mu_P x^n + \mu_{n-1} x^{n-1} + \cdots + \mu_0,
$$
with $\mu_{n-i}\in K[y]$ and $\deg(\mu_{n-i})\le i$. Let $\phi$ be the automorphism of
$R$ defined by $\phi(y) := y$ and $\phi(x) := x - \cramped{\frac{\mu_{n-1}}{n}}$.
Replacing $P$ and $Q$ by $\cramped{\frac{1}{\mu_P} \phi(\varphi_{\lambda}(P))}$ and
$\cramped{\frac{1}{\mu_Q} \phi(\varphi_{\lambda}(Q))}$, respectively, we can assume
without loss of generality that
\begin{equation}
P = x^n + \gamma_{n-2} x^{n-2} + \cdots + \gamma_0 \quad\text{and}\quad Q =
x^m + \delta_{m-1} x^{m-1} + \cdots + \delta_0,\label{eqq1}
\end{equation}
with $\gamma_{n-i},\delta_{m-i}\in K[y]$ and $\deg(\gamma_{n-i}),\deg(\delta_{m-i})\le i$.
Furthermore, a standard straightforward computation shows that there exists a unique
$C\in K[y]((x^{-1}))$ such that
\begin{equation}
C^n = P\qquad\text{and}\qquad C = x + C_0 + C_{-1} x^{-1}+ C_{-2} x^{-2} + \cdots,\label{eqq2}
\end{equation}
where $C_k\in K[y]$, $C_0=0$ and $\deg_y(C_k)\le -k+1$ for all $k\le -1$. It is easy to
see that $C$ is invertible and
$$
C^j = x^j + (C^j)_{j-1} x^{j-1} + (C^j)_{j-2} x^{j-2} + (C^j)_{j-3} x^{j-3}
+ (C^j)_{j-4} x^{j-4} + \cdots\qquad\text{for all $j\in \mathds{Z}$,}
$$
where $(C^j)_{-k}\in K[y]$, $(C^j)_{j-1}=0$ and $\deg_y((C^j)_k)\le -k+j$ for all $k\le j-2$.

\begin{definition} Let $H = \sum a_{ij} x^i y^j \in K[y]((x^{-1}))\setminus\{0\}$. The
{\em support} of $H$ is
$$
\Supp(H):=\left\{(i,j)\in \mathds{Z}\times\mathds{N}_0 : a_{ij}\ne 0 \right\}.
$$
\end{definition}

Let $\mathbf{w}=(w_1,w_2)$ be a direction. For $H\in K[y]((x^{-1}))\setminus\{0\}$, we write
$$
\mathbf{w}\!\deg(H):=\sup\{iw_1 + jw_2 : (i,j)\in \Supp(H)\}.
$$
Of course it is possible that $\mathbf{w}\!\deg(H) = +\infty$.

\smallskip

For $P,Q\in K[y]((x^{-1}))$ we define
$$
[P,Q] := \partial_x P \partial_y Q -  \partial_y P \partial_x Q,
$$
where $\partial_x P$ denotes the formal derivative of $P$ with respect to $x$, etcetera. It is easy to see that
$$
\mathbf{w}\!\deg([P,Q]) \le \mathbf{w}\!\deg(P) + \mathbf{w}\!\deg(Q) - |\mathbf{w}|,
$$
for any direction $\mathbf{w}$.

\begin{definition}[{\cite{vdE}*{page 247}}]\label{extension} Let $P$ be a polynomial of
degree $>1$ having a Jacobian mate of degree $>1$ and let $\mathbf{w}$ be a direction.
Let $R[P_+^{-1}]$ be the localization of $R$ in $P_+$. The ring extension $\tilde R_{P_+}$
of $R$ is the set of formal sums $f := \sum_{i\in \mathds{Z}} f_i$, where each $f_i$ is a
$\mathbf{w}$-homogeneous element of $R[P_+^{-1}]$ of degree $i$ and $f_i = 0$ for $i\gg 0$.
If $f\ne 0$, then the highest $i$ with $f_i\ne 0$, denoted by $\mathbf{w}\!\deg(f)$, is
called the {\em $\mathbf{w}$-degree of $f$}, while $f_i$ is denoted by $f_+$.
\end{definition}

\begin{proposition}\label{calculo de R} If $\mathbf{w}=(1,1)$ and $P$ is as in~\eqref{eqq1},
then $\tilde R_{P_+}$ is in a natural way a graded subalgebra of $K[y]((x^{-1}))$.
\end{proposition}

\begin{proof} Write
$$
P_+ = x^n + \alpha_1 yx^{n-1} + \alpha_2 y^2x^{n-2} +\cdots + \alpha_n y^n = x^n - B
$$
where $\alpha_1,\dots,\alpha_n\in K$ and
$B := - \alpha_1 yx^{n-1} - \alpha_2 y^2x^{n-2} -\cdots - \alpha_n y^n$ (actually $\alpha_1=0$
but we do not use this fact). A direct computation shows that $P_+$ is invertible in
$K[y]((x^{-1}))$ and that
$$
P_+^{-1} = x^{-n} + x^{-2n}B + x^{-3n}B^2 + x^{-4n}B^3 + \cdots.
$$
Note that the sum in the right side of this equality is well defined since
$$
\deg_x(x^{-in-n} B^i) \le (n-1)i - in-n =-n-i.
$$
In order to finish the proof it suffices to show that each series
$$
\sum_{i\le r} f_i\qquad\text{with $f_i\in K[y]((x^{-1}))$ such that $\deg(f_i) = i$,}
$$
is summable in $K[y]((x^{-1}))$. But this follows from the fact that $\deg(f_i) = i$ implies that
$$
f_i = \beta_0 x^i + \beta_1 x^{i-1} + \beta_2 x^{i-2} +\cdots
$$
with $\beta_i\in K[y]$ and $\deg(\beta_i)\le i$.
\end{proof}

In order to prove Theorem~\ref{principal}, we will need to use the following result, in
which $P_+$ and $F_+$ are taken with respect to the $(1,0)$-grading.

\begin{lemma}\label{auxiliar} Let $P,F\in K[y]((x^{-1}))$ be such that $P_+ = x^n$,
$\deg_x(F)\le 1-n$ and $[P,F]\in K^{\times}$. Then $F_+=(\mu_0+\mu_1 y)x^{1-n}$ with $\mu_1\ne 0$.
\end{lemma}

\begin{proof} Let $P=\sum_{i\le n} P_{i}$ and $F=\sum_{j\le 1-n} F_{j}$ be the
$(1,0)$-homogeneous decompositions of $P$ and $F$. Then the $(1,0)$-homogeneous decomposition of
$$
[P,F] = [P,F]_0 + [P,F]_{-1} + [P,F]_{-2} +\cdots
$$
is given by
$$
[P,F]_k=\sum_{i+j=k+1} [P_i,F_j].
$$
Write $F_{1-n} = x^{1-n} f_{1-n}(y)$. Since $[P,F]\in K^{\times}$, we have
$$
nf'(y) = [x^n,x^{1-n}f_{1-n}(y)] = [P_n,F_{1-n}] = [P,F]_0 = [P,F]\in K^{\times}
$$
So $f'(y)\in K^{\times}$, which implies that $f(y)=\mu_0+\mu_1 y$ for some
$\mu_0\in K$ and $\mu_1\in K^{\times}$, as desired.
\end{proof}

We also will need the following particular case of~\cite{vdE}*{Lemma~10.2.11}:

\begin{proposition}\label{vdE 10.2.11} Let $\mathbf{w}=(1,1)$ and let $P$ be as
in~\eqref{eqq1} and $C\in K[y]((x^{-1}))$ as in~\eqref{eqq2}. Assume $P$ has a
Jacobian mate $Q\!\in\! R$ of degree~$>1$ and let $\tilde Q\!\in\!\tilde R_{P_+}$
be such that $[P,\tilde Q]\!\in\! K^{\times}$. If
$$
\deg(P)+\deg(\tilde Q)-2 > 0,
$$
then there exists $j\in\mathds{Z}$ and $\lambda\in K^{\times}$ such that
$C^j\in \tilde R_{P_+}$ and $\deg(\tilde Q-\lambda C^j)<\deg(\tilde Q)$.
\end{proposition}

\begin{remark} The number $n$ that appears in the statement
of~\cite{vdE}*{Lemma~10.2.11} is not the degree of $P$, but only a divisor of $\deg(P)$.
The element $P^{\frac{1}{n}}$, introduced in~\cite{vdE} above of Lemma~10.2.10, equals
$\mu C^{\deg(P)/n}$ where $\mu\in K^{\times}$ and $n$ is as in~\cite{vdE}*{Lemma~10.2.11}.
\end{remark}

\begin{theorem}\label{principal} The JC is false if and only if there exist

\begin{itemize}

\smallskip

\item[-] $P,Q\in R$ and $C,F\in K[y]((x^{-1}))$,

\smallskip

\item[-] $n,m\in \mathds{N}$ such that $n\nmid m$ and $m\nmid n$,

\smallskip

\item[-] $\lambda_i\in K$ ($i=0,\dots,m+n-2$) with $\lambda_0=1$,

\smallskip

\end{itemize}
such that
\begin{itemize}

\smallskip

\item[-] $C$ has the form
$$
C = x + C_{-1}x^{-1}+ C_{-2}x^{-2} + \cdots \qquad\text{with each $C_{-i}\in K[y]$,}
$$

\smallskip

\item[-] $\deg(C)=1$ and $\deg(F)=2-n$,

\smallskip

\item[-] $F_+=x^{1-n}y$, where $F_+$ is taken with respect to the $(1,0)$-grading,

\smallskip

\item[-] $C^n=P$ and $Q=\sum_{i=0}^{m+n-2}\lambda_i C^{m-i}+F$.

\smallskip

\end{itemize}
Furthermore, under these conditions, $(P,Q)$ is a counterexample to the Jacobian conjecture.

\end{theorem}

\begin{proof} $\Rightarrow$)\enspace By Proposition~\ref{contraejemplo subrectangular}
we know that there exists a Jacobian pair $(P,Q)$ that is an counterexample, such that
neither $n\nmid m$ nor $m\nmid n$, where $n:=\deg(P)$ and $m:=\deg(Q)$. Futhermore,
by the discusion below that proposition, we can assume that $P$ and $Q$ are as
in~\eqref{eqq1}. Let $C$ be as in \eqref{eqq2}. Thus $\deg(C) = 1$, $C^n = P$ and $C$
has the form
$$
C = x + C_{-1}x^{-1}+ C_{-2}x^{-2} + \cdots \qquad\text{with each $C_{-i}\in K[y]$.}
$$
Since $m+n>2$, by Proposition~\ref{vdE 10.2.11} there exist $j\in\mathds{Z}$ and
$\lambda\in K^{\times}$ such that
$$
\deg(Q-\lambda C^j)<\deg(Q).
$$
By~\eqref{eqq1} and~\eqref{eqq2}, we have $j=m$ and $\lambda = 1$. We claim that there exist
 $\lambda_1,\dots,\lambda_{m+n-3}\in K$ such that
\begin{equation}
\deg\bigl(Q - C^m - \lambda_1 C^{m-1}-\cdots-\lambda_{m+n-3} C^{3-n}\bigr)\le 2-n.\label{eqqq5}
\end{equation}
Assume he have found $\lambda_1,\dots,\lambda_i\in K$, where $i<m+n-2$, such that
\begin{equation}
\deg(Q-C^m-\lambda_1C^{m-1}-\cdots-\lambda_i C^{m-i})\le m-i-1\label{eqqq3}
\end{equation}
Let $\tilde Q:= Q-C^m-\lambda_1C^{m-1}-\cdots-\lambda_i C^{m-i}$. If
$$
n+\deg(\tilde Q)-2 = \deg(P)+\deg(\tilde Q)-|(1,1)| \le 0,
$$
then we take $\lambda_{i+1} = \cdots = \lambda_{m+n-3} = 0$. Otherwise,
\begin{equation}
\deg(\tilde Q) > 2 - n,\label{eqqq4}
\end{equation}
and, again by Proposition~\ref{vdE 10.2.11}, there exist $j\in \mathds{Z}$ and $\lambda_j\in K^{\times}$
such that
$$
\deg(\tilde Q - \lambda_j C^{m-j})<\deg(\tilde Q).
$$
Consequently,
$$
m-j = \deg(C^{m-j}) = \deg(\tilde Q),
$$
and so, by~\eqref{eqqq3} and~\eqref{eqqq4},
$$
i+1\le j\le m+n-3.
$$
This finishes the proof of the claim. Let
$$
\tilde{F}:=Q - C^m - \lambda_1 C^{m-1}-\dots-\lambda_{m+n-3} C^{3-n}.
$$
Since $\deg(\tilde{F})\le 2-n$, there exist $\tilde{F}_0,\tilde{F}_1,\dots \in K[y]$ with
$\deg(\tilde{F}_i)\le i$, such that
$$
\tilde{F} = \tilde{F}_0 x^{2-n} + \tilde{F}_1 x^{1-n}+ \tilde{F}_2 x^{-n}+\cdots.
$$
Setting $\lambda_{m+n-2}:=\tilde{F}_0$ we obtain that
\begin{equation} \label{q como serie de C}
Q= C^m + \lambda_1 C^{m-1}+\dots+\lambda_{m+n-3} C^{3-n}+\lambda_{m+n-2} C^{2-n}+F,
\end{equation}
where
\begin{equation}
F := \tilde{F} - \lambda_{m+n-2} C^{2-n} = F_1 x^{1-n}+F_2 x^{-n}+F_3 x^{-n-1} +\cdots,\label{eqq3}
\end{equation}
where $F_i\in K[y]$ and $\deg(F_i)\le i$. Hence $\deg_x(F)\le 1-n$ and $F_1=\mu_0+\mu_1 y$
with $\mu_0,\mu_1\in K$. Moreover since $P = C^n$ we have $[P,F] = [P,Q]\in K^{\times}$
and so, $\mu_1\ne 0$, by Lemma~\ref{auxiliar}. Let $\varphi$ be the automorphism of
$K[y]((x^{-1}))$ defined by
$$
\varphi(x):=x\quad\text{and}\quad \varphi(y):=\frac{y-\mu_0}{\mu_1}.
$$
Replacing $P$, $Q$, $C$ and $F$ by $\varphi(P)$, $\varphi(Q)$, $\varphi(C)$ and $\varphi(F)$,
respectively, we can assume $\mu_0=0$ and $\mu_1=1$. Thus $F_+=x^{1-n}y$, where $F_+$ is
taken with respect to the $(1,0)$-grading. Note that this equality, combined with the
fact that $\deg(F_i)\le i$ for all $i$, gives $\deg(F) = 2-n$.

\smallskip

\noindent $\Leftarrow$)\enspace Since
$$
[P,F] - [P-P_+,F] - [P_+,F-F_+] = [P_+,F_+] = [x^n,x^{1-n}y] = n,
$$
where $P_+$ and $F_+$ are taken with respect to the $(1,0)$-grading, and
$$
\deg_x([P\!-\!P_+,F]),\deg_x([P_+,F\!-\!F_+])<\deg_x(P)\!+\!\deg_x(F)\!-\!1\le \deg(P)\!+\!\deg_x(F)\!-\!1 =0,
$$
we have
$$
[P,F] = n + \text{terms with $\deg_x$ lesser that $0$.}
$$
Moreover, using that
$$
C^n=P\qquad\text{and}\qquad Q=\sum_{i=0}^{m+n-2}\lambda_i C^{m-i}+F,
$$
we obtain that $\deg(P) = n$, $\deg(Q) = m$ and $[P,Q] = [P,F]$. Hence, neither $\deg(P)$
divides $\deg(Q)$ nor $\deg(Q)$ divides $\deg(P)$ and, since $[P,Q]\in R$, we also have
$$
[P,Q] = [P,F] = n\in K^{\times}.
$$
Consequently, by Proposition~\ref{contraejemplo subrectangular} the JC is false.
\end{proof}

\begin{remark}\label{re 1.9} The proof of the theorem shows that if $(P,Q)$ is a Jacobian
pair such that neither $\deg(P)$ divides $\deg(Q)$ nor $\deg(Q)$ divides $\deg(P)$, then
there is an affine change of variables that transforms it into a pair that satisfies the
conditions of the statement of Theorem~\ref{principal}. Note that a such change of
variables does not change neither $\deg(P)$ nor $\deg(Q)$.
\end{remark}

\begin{definition}\label{solucion de sistema} Let $D$ be a $K$-algebra, $n,m\in\mathds{N}$
such that $n\nmid m$ and $m\nmid n$, $(\lambda_i)_{1\le i\le n+m-2}$ a family of elements
of $K$ with $\lambda_0=1$ and $F_{1-n}\in D$. We say that $C \in D((x^{-1}))$ is a
{\em solution} of the system $S(n,m,(\lambda_i),F_{1-n})$, if $C$ has the form
$$
C = x + C_{-1}x^{-1}+ C_{-2}x^{-2} + \cdots \qquad\text{with each $C_{-i}\in D$,}
$$
and there exist $P,Q\in D[x]$ and $F \in D[[x^{-1}]]$, such that
\begin{align}
& F = F_{1-n} x^{1-n} + F_{-n} x^{-n} + F_{-1-n} x^{-1-n} +\cdots,\label{forma de F}\\
&P=C^n\qquad\text{and}\qquad Q=\sum_{i=0}^{m+n-2}\lambda_i C^{m-i}+F.\label{forma de P y de Q}
\end{align}
\end{definition}

Note that the polynomial $Q$ does not depend on $F$ since $\deg_x(F)\!<0$. We say that $(P,Q)$
is the {\em pair associated with the solution $C$} and we call $P,Q$ {\em the po\-ly\-no\-mials
associated with the solution $C$}.

\smallskip

From now on, when we mention a system $S(n,m,(\lambda_i),F_{1-n})$, unless otherwise specified,
we will assume that $n\nmid m$ and $m\nmid n$.

\begin{corollary}\label{corolario} The Jacobian conjecture is false if and only if for $D:=K[y]$
there exist

\begin{itemize}

\smallskip

\item[-] $n,m\in \mathds{N}$, such that $n\nmid m$ and $m\nmid n$,

\smallskip

\item[-] a family $(\lambda)_{0\le i \le m+n}$ of elements of $K$ with $\lambda_0=1$,

\smallskip

\item[-] a solution $C\in D((x^{-1}))$ of $S(n,m,(\lambda_i),y)$ such that
$$
\qquad\quad \deg(C) = 1\quad\text{and}\quad \deg(F) = 2-n,
$$
where $F$ is as in Definition~\ref{solucion de sistema}.

\smallskip

\end{itemize}
\end{corollary}

Let $A$ be an arbitrary $K$-algebra. In the sequel for each $E\!\in\! A((x^{-1}))$ and
$k\!\in\! \mathds{Z}$ we let $E_k$ denote the coefficient of $x^k$ in $E$.

\begin{remark}\label{re 1.21} Let $S(n,m,(\lambda_i),F_{1-n})$ be as in
Definition~\ref{solucion de sistema} and let $A$ be the polynomial
$K$-algebra $D[Z_{-1},Z_{-2},Z_{-3},\dots]$ in the indeterminates $Z_v$, with $v<0$.
Consider the Laurent series
$$
Z := x+Z_{-1}x^{-1}+Z_{-2}x^{-2}+\cdots\in A((x^{-1})).
$$
If $C\!\in\! D((x^{-1}))$ is a solution of $S(n,m,(\lambda_i),F_{1-n})$, then the
coefficients $C_{-1},\dots,C_{-m-n+2}$ satisfy the $m+n-2$ equations
\begin{equation}
\begin{aligned}
(Z^n)_{-k} & = 0, && \text{for $k=1,\dots, m-1$,}\\
\left(\sum_{i=0}^{m+n-2}\lambda_i Z^{m-i}\right)_{-k} & = 0, &&\text{for $k=1,\dots,n-2$,}\\
\left(\sum_{i=0}^{m+n-2}\lambda_i Z^{m-i}\right)_{1-n} + F_{1-n} & = 0.
\end{aligned}\label{sistema de ecuaciones}
\end{equation}
(Note that $Z_{-n-m+2}$ is the the lowest degree coefficient of $Z$ which appears
in the system. It appears in the equation $(Z^n)_{1-m}=0$ and in the last equation).

\smallskip

Conversely, if $C_{\!-1},\dots,C_{\!-m-n+2}\!\in\! D$ satisfy the equation
system~\eqref{sistema de ecuaciones}, then there exist u\-ni\-que
$$
C_{-m-n+1},C_{-m-n}, C_{-m-n-1},\dots\in D,
$$
such that
\begin{equation}
C := x + C_{-1}x^{-1}+ C_{-2}x^{-2}+ C_{-3}x^{-3} + \cdots \label{solucion asociada}
\end{equation}
is a solution of $S(n,m,(\lambda_i),F_{1-n})$. In fact, let $j\in \mathds{N}_0$ and
assume we have proven that there exist unique
$$
C_{-m-n-i+2}\in D \qquad\text{where $i$ runs from $1$ to $j$,}
$$
such that $C_{-1},\dots, C_{-m-n-j+2}$ satisfy
\begin{equation}
(Z^n)_{-k} = 0\qquad\text{for $k=1,\dots,m-1+j$.}\label{eqqq1}
\end{equation}
Since
$$
(Z^n)_{-m-j} = H + n Z_{-m-n-j+1},
$$
where $H$ is a sum of monomials of $K[Z_{-1},\dots,Z_{-m-n-j+2}]$, we can solve
$Z_{-m-n-j+1}$ univocally in the equation
$$
0 = H(C_{-1},\dots,C_{-m-n-j+2}) + nZ_{-m-n-j+1}.
$$
So, there exists a unique $C_{-m-n-j+1}\in D$ such that $C_{-1},\dots, C_{-m-n-j+1}$ satisfy
\begin{equation*}
(Z^n)_{-m-j} = 0.
\end{equation*}
It is evident that $\bigl(C_{-1},\dots, C_{-m-n-j}\bigr)$ satisfies the system of
equations~\eqref{eqqq1}, since $Z_{-m-n-j+1}$ does not appear in that system. In
order to finish the proof we only must note that
$$
F = F_{1-n} x^{1-n} + F_{-n} x^{-n} + F_{-1-n} x^{-1-n} +\cdots,
$$
is univocally determined by the equations
$$
\left(\sum_{i=0}^{m+n-2}\lambda_i Z^{m-i}\right)_{-k} + F_{-k}  = 0\qquad\text{for $k\ge n$.}
$$
\end{remark}

\begin{definition}\label{sistema de ecuaciones asociado con} We will write
$S_t(n,m,(\lambda_i),F_{1-n})$ to denote the system of equations~\eqref{sistema de ecuaciones},
and we call it the {\em (standard) system of equations associated with} $S(n,m,(\lambda_i),F_{1-n})$.
\end{definition}

\begin{definition}\label{serie determinada por} Given a solution $C_{-1},\dots,C_{-m-n+2}\in D$
of~\eqref{sistema de ecuaciones}, we call~\eqref{solucion asociada} the
{\em solution of $S(n,m,(\lambda_i),F_{1-n})$ determined by $C_{-1},\dots,C_{-m-n+2}$}.
\end{definition}

\begin{remark}\label{acotacion de grado} Assume that $D = K[y]$. Let $S(n,m,(\lambda_i),y)$
be as in Corollary~\ref{corolario} and let
$$
C = x+C_{-1}x^{-1}+C_{-2}x^{-2}+C_{-3}x^{-3}+\dots \in D((x^{-1}))
$$
be a solution of $S(n,m,(\lambda_i),y)$. Note that for $j>-m$,
\begin{align*}
0 &=(C^n)_{-m-j}\\
&=\sum_{i_1+\dots +i_n=-m-j}C_{i_1}C_{i_2}\dots C_{i_n}\\
&= nC_{-m-n-j+1}+ \sum_{{i_1+\dots +i_n=-m-j}\atop{i_k\ne -m-n-j+1\ \forall k}}C_{i_1}C_{i_2}\dots C_{i_n},
\end{align*}
where we set $C_1 = 1$. From this it follows by induction that if
\begin{equation}\label{desigualdad en deg C_{-k} para algunos}
\deg(C_{-k})\le k+1\qquad\text{for $k=1,2,\dots,m+n-2$,}
\end{equation}
then
\begin{equation}\label{desigualdad en deg C_{-k} para todos}
\deg(C_{-k})\le k+1 \qquad\text{for all $k\ge 1$.}
\end{equation}
Note also that equality in~\eqref{desigualdad en deg C_{-k} para algunos} implies
equality in~\eqref{desigualdad en deg C_{-k} para todos}.  A similar argument proves that
under the same hypothesis,
$$
\deg(F_{-k})\le 2-n+k \qquad\text{for all $k\ge n$.}
$$
\end{remark}

Resuming the results of this section we have the following corollary.

\begin{corollary}\label{caracterizacion} There exists a counterexample $(P,Q)$ to JC with
$(\deg(P),\deg(Q))=(m,n)$ if and only if there exist $\lambda_1,\dots, \lambda_{m+n-2}\in K$
such that the standard system  $S_t(n,m,(\lambda_i),y)$ has a solution in $K[y]^{m+n-2}$.
\end{corollary}

\section[Properties of solutions of $S(n,m,(\lambda_i),F_{1-n})$]{Properties of solutions of
$\bm{S(n,m,(\lambda_i),F_{1-n})}$}
In this section we show that under suitable conditions the system~$S(n,m,(\lambda_i),F_{1-n})$
has only finitely many solutions. This applies in particular to the case related with the
Jacobian conjecture.

\setcounter{equation}{0}

\begin{lemma} Let $Z$ be as in Remark~\ref{re 1.21}. For all $i\in \mathds{N}$ and
$k,l\in \mathds{Z}$, the equality
$$
\frac{\partial (Z^i)_{k}}{\partial Z_l}=i (Z^{i-1})_{k-l}
$$
holds.
\end{lemma}

\begin{proof} Since
$$
\sum_k \frac{\partial (Z^i)_k}{\partial Z_l} x^k=\frac{\partial \bigl(\sum_k (Z^i)_k x^k\bigr)}{\partial Z_l}=
 \frac{\partial Z^i}{\partial Z_l}=i Z^{i-1}\frac{\partial Z}{\partial Z_l}=i Z^{i-1}x^l= \sum_j i(Z^{i-1})_j
 x^{j+l},
$$
we have
$$
\frac{\partial (Z^i)_{k}}{\partial Z_l}=i (Z^{i-1})_{k-l},
$$
as desired.
\end{proof}

Let $S(n,m,(\lambda_i),F_{1-n})$ be as in Definition~\ref{solucion de sistema}. Let $Z$ and $A$
be as in Remark~\ref{re 1.21}. Consider the polynomials
$$
E_1,\dots,E_{m+n-2}\in D[Z_{-1},\dots,Z_{-m-n+2}],
$$
defined by
$$
E_i := \begin{cases} (Z^n)_{-i} &\text{for $1\le i < m$,}\\[5pt]
\left(\sum_{k=0}^{m+n-2}\lambda_k Z^{m-k}\right)_{m-i-1} &\text{for $m\le i< m+n-2$,}\\[5pt]
\left(\sum_{k=0}^{m+n-2}\lambda_k Z^{m-k}\right)_{1-n}+F_{1-n}&\text{for $i=m+n-2$,}\end{cases}
$$
and set
$$
J:= \begin{pmatrix}
\frac{\partial E_1}{\partial Z_{-1}} &\dots & \frac{\partial E_1}{\partial Z_{-m-n+2}}\\
\vdots & \ddots & \vdots\\
\frac{\partial E_{m+n-2}}{\partial Z_{-1}} &\dots & \frac{\partial E_{m+n-2}}{\partial Z_{-m-n+2}}\\
\end{pmatrix}.
$$
Note that since $J$ is a matrix in $D[Z_{-1},\dots,Z_{-m-n+2}]$ it makes sense to evaluate it
in the tuple $\bigl(C_{-1},\dots,C_{-m-n+2}\bigr)$. Let
\begin{equation}
G := \sum_{k=0}^{m+n-2}\lambda_k (m-k) Z^{m-k-1}.\label{eq6}
\end{equation}
By the previous lemma we know that
$$
\frac{\partial E_i}{\partial Z_{-j}} = \begin{cases} n(Z^{n-1})_{j-i} & \text{for $1\le i < m$,}\\
G_{m+j-i-1} & \text{for $m\le i< m+n-1$.}\end{cases}
$$
Since
$$
\deg(Z^{n-1}) = n-1 \quad\text{and}\quad \deg(G) = m-1,
$$
this implies that $J$ is the matrix $(Y_{ij})\in M_{m+n-2}(A)$  given by
$$
Y_{ij} : = \begin{cases} n(Z^{n-1})_{j-i} & \text{if $1\le i< m$ and $1\le j< n+i$,}\\
G_{m+j-i-1} & \text{if $m\le i < m+n-1$ and $1\le j\le i$,} \\ 0 & \text{otherwise.}\end{cases}
$$
In other words
\begin{equation}\label{eq3}
J = \begin{pmatrix}
J^{(1)}\\  J^{(2)}
\end{pmatrix},
\end{equation}
where
\begin{equation}
J^{(1)}\in M_{(m-1)\times(m+n-2)}(A)\quad\text{and}\quad J^{(2)}\in M_{(n-1)\times(m+n-2)}(A)\label{eq4}
\end{equation}
are the matrices
$$
J^{(1)} := \begin{pmatrix} n(Z^{n-1})_{0} & \dots & n(Z^{n-1})_{n-1} & & \BigZero\\
\vdots &\ddots &\vdots &\ddots \\
n(Z^{n-1})_{2-m}  &\dots & n(Z^{n-1})_{n-m+1} &\dots & n(Z^{n-1})_{n-1}\\[5pt]
\end{pmatrix}
$$
and
$$
J^{(2)} := \begin{pmatrix} G_{0} &\dots & G_{m-1} && \BigZero\\
\vdots &\ddots &\vdots &\ddots \\
G_{2-n} &\dots & G_{m-n+1} &\dots & G_{m-1}
\end{pmatrix},
$$
respectively.

\smallskip

For each $M\in M_{r\times s}(D)$, we let $\ov{M}$ denote the $k$-linear map, from $D^s$ to $D^r$, given by
$$
\ov{M}(V):= (M V^t)^t\in D^r,
$$
where, as usual, $X^t$ denotes the transpose of $X$. In order to prove
Theorem~\ref{Jinvertible} below, we need to introduce some auxiliary maps.

\begin{definition} We define the maps
\begin{align*}
&\Pi_1\colon D((x^{-1}))\to D^{m-1}&&\text{by} &&\Pi_1(f):=(f_{-1},f_{-2},\dots,f_{1-m}), \\
&\Pi_2\colon D((x^{-1}))\to D^{n-1}&&\text{by} &&\Pi_2(f):=(f_{-1},f_{-2},\dots,f_{1-n}),\\
&\Gamma_1\colon D^{m-1}\to D((x^{-1}))&&\text{by} &&\Gamma_1(d_1,\dots,d_{1-m}):=
d_1 x^{-1}+\dots+d_{1-m}x^{1-m},\\
&\Gamma_2\colon D^{n-1}\to D((x^{-1}))&&\text{by} &&\Gamma_2(d_1,\dots,d_{1-n}):=
d_1 x^{-1}+\dots+ d_{1-n} x^{1-n}.
\end{align*}
\end{definition}

Note that $\Gamma_1$ and $\Gamma_2$ are right inverses to $\Pi_1$ and $\Pi_2$,
respectively. We will also need the map
$$
\Pi\colon D((x^{-1}))\longrightarrow D^{m+n-2},
$$
defined by $\Pi(f) := (f_{-1},f_{-2},\dots,f_{-m-n+2})$, and the canonical projections
$$
\Pi_+\colon D((x^{-1}))\to D[x]\qquad\text{and}\qquad \Pi_-\colon D((x^{-1}))\to D[[x^{-1}]].
$$

\begin{theorem}\label{Jinvertible} Let $S(n,m,(\lambda_i),F_{1-n})$ be as in
Definition~\ref{solucion de sistema},
$$
C = x + C_{-1} x^{-1} +  C_{-2} x^{-2} + \cdots \in D((x^{-1}))
$$
a solution of $S(n,m,(\lambda_i),F_{1-n})$ and $(P,Q)$ the pair associated with $C$.
Assume that $D$ is an integral domain over $K$. If there exist $A,B\in D[x]$ such
that $AP'+BQ'=1$, then the matrix $J_{|_{\mathfrak{v}}}$, obtained evaluating $J$ in
$$
\mathfrak{v}:=\bigl(C_{-1},\dots,C_{-m-n+2}\bigr)\in D^{m+n-2},
$$
is invertible.
\end{theorem}

\begin{proof} Recall that $P = C^n$ and that there exists $F\in D((x^{-1}))$ such that
\begin{equation}\label{formula de F}
F = F_{1-n} x^{1-n} + F_{-n} x^{-n} + F_{-1-n} x^{-1-n} +\cdots,
\end{equation}
and
$$
Q = \sum_{i=0}^{m+n-2}\lambda_i C^{m-i}+F.
$$
Let $G$ be as in~\eqref{eq6}. Note that
$$
G(C) = \sum_{i=0}^{m+n-2}\lambda_i (m-i) C^{m-i-1}
$$
satisfies
\begin{equation}
G(C)C' = Q'-F'.\label{eq14}
\end{equation}
We claim that if $V\in D^{m+n-2}$ and $U\in D((x^{-1}))$ satisfy $\deg_x(U)<0$ and $\Pi(U)=V$, then
\begin{equation}
\ov{J^{(1)}_{|_{\mathfrak{v}}}}(V) = \Pi_1(nC^{n-1} U)\qquad\text{and}\qquad
\ov{J^{(2)}_{|_{\mathfrak{v}}}}(V)=\Pi_2(G(C)U).\label{eq15}
\end{equation}
Let $\ov{J^{(1)}_{|_{\mathfrak{v}}}}(V)_i$ be the $i$-th coordinate of $\ov{J^{(1)}_{|_{\mathfrak{v}}}}(V)$.
Write $V = (v_1,\dots,v_{m+n-2})$. Since
$$
\deg_x(C^{n-1}) = n-1\qquad\text{and}\qquad \deg_x(U)< 0,
$$
we have
\begin{align*}
\Pi_1(nC^{n-1} U)_i & =  \sum_{j\in \mathds{Z}} n(C^{n-1})_j U_{-j-i}\\
& = n(C^{n-1})_{1-i} U_{-1} + \cdots + n(C^{n-1})_{n-1} U_{1-n-i}\\
& = n(C^{n-1})_{1-i} v_1 + \cdots + n(C^{n-1})_{n-1} v_{n+i-1}\\
& = \ov{J^{(1)}_{|_{\mathfrak{v}}}}(V)_i,
\end{align*}
proving the first equality in the claim. The second one is similar.

\smallskip

Now, given $v_1\in D^{m-1}$ and $v_2 \in D^{n-1}$, we set
$$
V := (v_1,v_2)\in D^{m+n-2},\qquad V_1:=\Gamma_1(v_1)\qquad\text{and}\qquad V_2:=\Gamma_2(v_2).
$$
We are going to prove that $V\in  \ov{J_{|_{\mathfrak{v}}}}(D^{m+n-2})$. Define $h\in D[x]$ by
\begin{equation}
h:=\Pi_+(V_2 P'-V_1 Q').\label{eq10}
\end{equation}
Note that
\begin{equation}
\deg_x(h)\le \min(n-2,m-2)<m+n-2.\label{eq7}
\end{equation}
From $AP'+BQ'=1$ we obtain $AhP'+BhQ'=h$. Since the leading term of $P'$ is invertible, there exist unique
$T,A_1\in D[x]$ with
\begin{equation}
\deg_x(A_1)<\deg_x(P')=n-1,\label{eq8}
\end{equation}
such that $hB=TP'+A_1$. Let $A_2:=-Ah-TQ'$. A direct computation shows that
\begin{equation}
A_1 Q' - A_2 P' = h.\label{eq11}
\end{equation}
Using this equality, conditions~\eqref{eq7} and~\eqref{eq8}, and that $\deg_x(Q') = m-1$, we obtain that
\begin{equation}
\deg_x(A_2)<m-1.\label{eq9}
\end{equation}
Note that $nC^{n-1}$ and $G(C)$ have invertible leading terms and hence are invertible in $D((x^{-1}))$.
Moreover, by the definition of $\Gamma_1$ and~\eqref{eq8}, we have
\begin{equation}
\deg_x(A_1+V_1)<n-1,\label{eq12}
\end{equation}
which implies
\begin{equation}
\deg_x\left(\frac{A_1+V_1}{nC^{n-1}}\right)\le (n-2)-(n-1)=-1.\label{eq12,1}
\end{equation}
A similar computation gives
\begin{equation}
\deg_x\left(\frac{A_2+V_2}{G(C)} \right)\le -1.\label{eq12,2}
\end{equation}
On the other hand, by~\eqref{eq14},
$$
(A_1+V_1)G(C)C' = (A_1+V_1)(Q'-F') = A_1Q'+V_1Q'- A_1F'-V_1F',
$$
and, by the fact that $P=C^n$ and equality~\eqref{eq11},
$$
(A_2+V_2)nC^{n-1}C' = (A_2+V_2)P' = A_1Q'-h+V_2P'.
$$
So,
\begin{align*}
\Pi\left( \frac{A_2+V_2}{G(C)}-\frac{A_1+V_1}{nC^{n-1}}\right)&= \Pi\left(\frac{V_2 P'-V_1
Q'-h+(A_1+V_1)F'}{nC^{n-1}G(C)C'}\right)\\
&=\Pi\left(\frac{\Pi_-(V_2 P'-V_1 Q')+(A_1+V_1)F'}{nC^{n-1}G(C)C'}\right)\\
&=0,
\end{align*}
where the second equality follows from~\eqref{eq10} and the last one from the facts that
$$
\deg_x(nC^{n-1}G(C)C')=(n-1)+(m-1)=m+n-2.
$$
and, by~\eqref{formula de F} and~\eqref{eq12},
$$
\deg_x((A_1+V_1)F')\le -1.
$$
We set
$$
X:=\Pi\left(\frac{A_2+V_2}{G(C)}\right)=\Pi\left(\frac{A_1+V_1}{nC^{n-1}}\right).
$$
Then, by~\eqref{eq12,1}, \eqref{eq12,2}, \eqref{eq15} and the fact that $A_1$ and $A_2$
are polynomials, we have
$$
\ov{J^{(1)}_{|_{\mathfrak{v}}}}(X)=\Pi_1\left(\frac{A_1+V_1}{nC^{n-1}}nC^{n-1}\right)=\Pi_1(A_1+V_1)=\Pi_1(V_1)=v_1
$$
and
$$
\ov{J^{(2)}_{|_{\mathfrak{v}}}}(X)=\Pi_2\left(\frac{A_2+V_2}{G(C)}G(C)\right)=\Pi_2(A_2+V_2)=\Pi_2(V_2)=v_2,
$$
which finishes the proof.
\end{proof}

\begin{corollary}\label{Jinvertiblecoro} Assume that we are under the hypothesis of
Theorem~\ref{Jinvertible} and that $D$ is an integral domain. Then $S(n,m,(\lambda_i),F_{1-n})$
has finitely many solutions.
\end{corollary}

\begin{proof} Let $L$ be an algebraic closure of the field of fractions of $D$. By the Jacobian Criterion,
applying Theorem~\ref{Jinvertible} with $D$ replaced by $L$, we obtain that the set of solutions
$$
(C_{-1},\dots,C_{-m-n-+2})\in L^{m+n-2}
$$
of $S(n,m,(\lambda_i),F_{1-n})$ is a zero-dimensional algebraic variety, and hence finite.
\end{proof}

\begin{remark} If $F_{1-n}=y$ and $(P,Q)$ is a counterexample to JC, then the hypothesis of
Theorem~\ref{Jinvertible} are fulfilled with $A:=\frac{\partial Q}{\partial y}$ and
$B:=-\frac{\partial P}{\partial y}$.
\end{remark}

\section[The homogeneous system $S(n,m,F_{1-n})$]{The homogeneous system $\bm{S(n,m,F_{1-n})}$}
\label{homogeneos}
\setcounter{equation}{0}
In this section we let $S(n,m,F_{1-n})$ denote the system $S(n,m,(\lambda_i),F_{1-n})$ with
$\lambda_i = 0$ for all $i\ne 0$, and we begin the study of the solution set of this system.
Consider the polynomials
$$
E^{(h)}_1,\dots,E^{(h)}_{m+n-2}\in K[Y][Z_{-1},\dots,Z_{-m-n+2}],
$$
defined by
\begin{equation}\label{Ehomogeneo}
E^{(h)}_i := \begin{cases} (\ov{Z}^n)_{-i} &\text{for $1\le i < m$,}\\[5pt] (\ov{Z}^m)_{m-i-1} &
\text{for $m\le i< m+n-2$,} \\[5pt] (\ov{Z}^m)_{1-n}+Y^{m+n-1} &\text{for $i=m+n-2$,}\end{cases}
\end{equation}
where
$$
\ov{Z} := x + Z_{-1} x^{-1} + \cdots + Z_{-m-n+2}x^{-m-n+2} \in K[Y][Z_{-1},\dots,Z_{-m-n+2}]((x^{-1})).
$$
Let $I^{(h)}$ be the ideal of $K[Y,Y^{-1}][Z_{-1},\dots,Z_{-m-n+2}]$ generated by the $E_i^{(h)}$'s.
Consider the weight $w$ on the variables given by $w(Z_{-k}):=k+1$ and $w(Y):=1$ and let $\wdeg$
denote the corresponding degree. Similar computations as in Remark~\ref{acotacion de grado} show
that each $E_i^{(h)}$ is $w$-homogeneous with
$$
\wdeg(E_i^{(h)})=\begin{cases}i+n &\text{if $i<m$,}\\ i+1 &\text{if $i\ge m$.}\end{cases}
$$
Propositions~\ref{proposicion} and~\ref{zannier} show that the system $S(n,m,Y^{m+n-1})$ has always a solution.
In the present section we do not need this result.

\begin{lemma}\label{polinomiounivariado} If there exists a solution
$$
C = x + C_{-1} x^{-1} + C_{-2} x^{-2} + C_{-3} x^{-3}+\cdots \in K[Y]((x^{-1}))
$$
of the system $S(n,m,Y^{m+n-1})$, then for $k\!=\!1,\dots,m\!+\!n\!-\!2$ there exist
$s_k\in \mathds{N}$ and a $w$-ho\-mo\-ge\-neous polynomial $h_k\in K[Y,Z_{-k}]\cap I^{(h)}$
with leading term $(Z_{-k})^{s_k}$, with respect to the graduation obtained giving weight
$1$ to $Z_{-k}$ and $0$ to $Y$.
\end{lemma}

\begin{proof} Throughout this proof we write $Z_{-k}^u$ instead of $(Z_{-k})^u$, and we let
$[R,S]$ denote the Jacobian $J_{x,Y}(R,S)$ with respect to the variables $x$ and $Y$. Let $P$
and $Q$ be the polynomials associated with $C$ and let $F$ be as in
Definition~\ref{solucion de sistema}. Let $P_+$ and $F_+$ be the leading terms of $P$ and $F$
with respect to $\deg_x$. Since, by~\eqref{forma de F} and~\eqref{forma de P y de Q},
\begin{equation}
[P_+,F_+] = [x^n,x^{1-n}Y^{m+n-1}] = n(m+n-1)Y^{m+n-2},\label{eqqq6}
\end{equation}
we have
\begin{equation}
\begin{aligned}
\deg_x([P-P_+,F])&\le \deg_x(P-P_+) + \deg_x(F) - 1\\
& < \deg_x(P) + \deg_x(F) - 1\\
& = n + (1-n) -1\\
& = \deg_x([P_+,F_+]),
\end{aligned}\label{eqqq7}
\end{equation}
and similarly,
\begin{equation}
\deg_x([P_+,F-F_+]) < \deg_x([P_+,F_+]).\label{eqqq8}
\end{equation}
Using~\eqref{eqqq6}, \eqref{eqqq7}, \eqref{eqqq8} and that
$$
[P_+,F_+] + [P-P_+,F] + [P_+,F-F_+] = [P,F],
$$
we obtain
\begin{equation}
[P,Q] = [P,F] = [P_+,F_+] = [x^n,x^{1-n}Y^{m+n-1}] = n(m+n-1)Y^{m+n-2},\label{eqqq9}
\end{equation}
where the first equality follows from~\eqref{forma de P y de Q}. Let $D$ be an algebraic closure of $K(Y)$.
By~\eqref{eqqq9}, if we set
$$
A := \frac{1}{n(m+n-1)Y^{m+n-2}}\,\frac{\partial Q}{\partial Y}\quad \text{and}\quad
B:=\frac{-1}{n(m+n-1) Y^{m+n-2}}\,\frac{\partial P}{\partial Y},
$$
then
$$
A P' + B Q' = 1.
$$
By theorem~\ref{Jinvertible} the set of all the solutions of the system of equations
$S_t(n,m,Y^{m+n-1})$, introduced in Definition~\ref{sistema de ecuaciones asociado con},
is finite. For each $k\in \{1,\dots,m+n-2\}$, let
$$
f:=\prod_{j=1}^r (Z_{-k}-a_j)\in D[Z_{-k}]\subseteq D[Z_{-1},\dots,Z_{-m-n+2}],
$$
where $\{a_1,\dots,a_r\}\subseteq D$ is the set formed by the $k$th coordinates of the
solutions in $D^{m+n-2}$ of the system of equations mentioned above. Let $\ov{I}^{(h)}$
be the extension of $I^{(h)}$ in $D[Z_{-1},\dots,Z_{-m-n+2}]$. By the nullstellensatz
$\cramped{f\in \sqrt{\ov I^{(h)}}}$, and so, there is $t\in \mathds{N}$ such that
$f^t\in \ov I^{(h)}$. This means that
\begin{equation}\label{una expresion para f^t}
f^t = \sum_i \hat{f}_i E^{(h)}_i,\qquad\text{for some $\hat{f}_i\in D[Z_{-1},\dots,Z_{-m-n+2}]$.}
\end{equation}
Let $K_1$ be the finite extension of $K(Y)$ generated by $a_1,\dots,a_r$. By the
definition of $f$ there exist $b_0,\dots,b_{rt-1}\in K_1$ such that
\begin{equation}\label{otra expresion para f^t}
f^t =  Z_{-k}^{rt}+b_{rt-1}Z_{-k}^{rt-1}+\dots+b_1 Z_{-k}+ b_0.
\end{equation}
Let $e_0,\dots,e_T$ be a basis of $K_1$ over $K(Y)$ with $e_0=1$. Write
$$
f^t = \sum_{l=0}^T h^{(l)} e_l ,\qquad \hat{f}_i = \sum_{l=0}^T f_i^{(l)} e_l\qquad
\text{and}\qquad b_j = \sum_{l=0}^T b_j^{(l)} e_l,
$$
where
$$
h^{(l)}\in K(Y)[Z_{-k}],\qquad f_i^{(l)}\in K(Y)[Z_{-1},\dots,Z_{-m-n+2}]\qquad\text{and}
\qquad b_j^{(l)}\in K(Y).
$$
Using~\eqref{una expresion para f^t},~\eqref{otra expresion para f^t}, that $\hat{f}_i =
\sum_{l=0}^T f_i^{(l)}e_l$ and that $e_0=1$, we obtain
\begin{equation}\label{calculo de sum f_i^{(0)} E^{(h)}_i} \sum_i f_i^{(0)} E^{(h)}_i =
Z_{-k}^{rt}+b_{rt-1}^{(0)}Z_{-k}^{rt-1}+\dots+b_1^{(0)} Z_{-k}+ b_0^{(0)}.
\end{equation}
Consider the canonical inclusion of $K(Y)$ into $K((Y^{-1}))$ and write
$$
b_i^{(0)} = \sum_{j\in \mathds{Z}} \lambda_{ij} Y^j\qquad\text{and}\qquad f_i^{(0)} = \sum_{j\in
\mathds{Z},\,\mathbf{l}\in \mathds{N}_0^{m+n-2}} \gamma_{j,\mathbf{l}} Y^j \mathbf{Z}^{\mathbf{l}},
$$
where
$$
\mathbf{Z}^{\mathbf{l}} := Z_{-1}^{l_1}\cdots Z_{-m-n+2}^{l_{m+n-2}}\qquad\text{if $\mathbf{l} =
(l_1,\dots,l_{m+n-2})$}.
$$
Set
$$
f_i := \begin{cases} \displaystyle{\sum_{(j,\mathbf{l})\in \mathcal{B}_{i+n}} \gamma_{j,\mathbf{l}} Y^j
\mathbf{Z}^{\mathbf{l}}} & \text{if $i<m$,}\\[13pt] \displaystyle{\sum_{(j,\mathbf{l})\in \mathcal{B}_{i+1}}
\gamma_{j,\mathbf{l}} Y^j \mathbf{Z}^{\mathbf{l}}} & \text{if $i\ge m$,}
\end{cases}
$$
where $\mathcal{B}_u:=\{(j,\mathbf{l}): \wdeg(Y^j \mathbf{Z}^{\mathbf{l}}) = rt(k+1)-u\}$. Note that $f_i$ is
the $w$-ho\-mo\-geneous component of $f_i^{(0)}$ satisfying
$$
\wdeg(f_i) + \wdeg\bigl(E_i^{(h)}\bigr)= rt(k+1) = \wdeg\bigl(Z_{-k}^{rt}\bigr)
$$
Taking the $w$-homogeneous component of degree $rt(k+1)$ in
equality~\eqref{calculo de sum f_i^{(0)} E^{(h)}_i}, we obtain
$$
\sum_i f_i E^{(h)}_i = Z_{-k}^{rt} + \sum_{j=1}^{rt} \lambda_{rt-j,jk+j} Y^{jk+j}  Z_{-k}^{rt-j}.
$$
and so $s_k:=rt$ and $h_k:=\sum_i f_i E^{(h)}_i$ satisfy the required conditions.
\end{proof}

\begin{theorem}\label{homogeneo} Assume that $C\in K[Y]((x^{-1}))$ is a solution of $S(n,m,Y^{m+n-1})$.
Then, for each $k=1,\dots,m+n-2$ there exists $c_{-k}\in K$ such that
$$
C_{-k}=c_{-k} Y^{k+1}.
$$
\end{theorem}

\begin{proof} Let $h_k(Z_{-k})\in K[Y][Z_{-k}]$ and $s_k$ be as in the previous lemma. Since $h_k$ is
$w$-homogeneous,
$$
h_k(Z_{-k}) = \sum_{i=r}^{s_k} \mu_i Y^{(s_k-i)(k+1)} Z_{-k}^i \qquad\text{with $\mu_r\ne 0$ and
$\mu_{s_k}=1$.}
$$
Since $h_k\in I^{(h)}$, we know that $h_k(C_{-k})=0$. Suppose $C_{-k}\ne 0$ and write
$$
C_{-k} = \sum_{j=t}^u \nu_j Y^j\qquad \text{with $\nu_t,\nu_u\in K^{\times}$}.
$$
In order to finish the proof we must check that $u = t= k+1$. But if $k+1<u$, then
$$
h_k(C_{-k})=\mu_{s_k} \nu_u^{s_k} Y^{u s_k}+\text{lower order terms},
$$
and consequently $h_k(C_{-k})\ne 0$, a contradiction. Similarly, if $t<k+1$, then
$$
h_k(C_{-k})= \mu_{s_k} \nu_t^{s_k} Y^{t s_k}+\text{higher order terms},
$$
and consequently again $h_k(C_{-k})\ne 0$.
\end{proof}

By Remark~\ref{re 1.21} from the solutions of $S_t(n,m,1)$ we obtain solutions of $S(n,m,1)$. In the next section
we will see that solutions of $S(n,m,1)$ determine solutions of $S(n,m,Y^{n+m-1})$.

\section[Presentations of the solutions of $S(n,m,Y^{m+n-1})$]{Presentations of the solutions of
$\bm{S(n,m,Y^{m+n-1})}$}

In this section we focus on solutions of the system $S(n,m,Y^{m+n-1})$. This system has many different
presentations. Note that if $(P,Q)$ is the pair associated with a solution of $S(n,m,Y^{m+n-1})$, then by
Theorem~\ref{homogeneo}  and Remark~\ref{acotacion de grado},
$$
P=C^n=\sum_{i=0}^n p_i x^i Y^{n-i}\qquad\text{and}\qquad Q=\Pi_{+}(C^m)=\sum_{i=0}^m q_i x^i Y^{m-i}
$$
are homogeneous polynomials, with $p_n = q_m =1$ and $p_{n-1} = q_{m-1} = 0$. Furthermore,
by~\eqref{eqqq9} we know that
$$
[P,Q] = n(m+n-1)Y^{m+n-2}.
$$

\begin{proposition}\label{proposicion} Let
$$
P=\sum_{i=0}^n p_i x^i Y^{n-i}\qquad\text{and}\qquad Q=\sum_{i=0}^m q_i x^i Y^{m-i}
$$
be homogeneous polynomials with $p_n = q_m = 1$ and $p_{n-1} = 0$. Define $p,q\in K[x]$~by
$$
p:=\sum_{i=0}^n p_i x^i\qquad\text{and}\qquad q:=\sum_{i=0}^m q_i x^i.
$$
Let $\lambda\in K^{\times}$ and set $\wt{\lambda}:=n\lambda(1-m-n)$. The following conditions are equivalent:
\begin{enumerate}

\smallskip

\item $(P,Q)$ is the pair associated with a solution
$$
C := x + C_{-1} x^{-1} + C_{-2} x^{-2} +\cdots \in K[Y]((x^{-1}))
$$
of the system $S(n,m,\lambda Y^{m+n-1})$.

\smallskip

\item $(p,q)$ is the pair associated with a solution
$$
c := x + c_{-1} x^{-1} + c_{-2} x^{-2} +\cdots \in K((x^{-1}))
$$
of the system $S(n,m,\lambda)$.

\smallskip

\item $[P,Q]=\wt{\lambda} Y^{m+n-2}$.

\smallskip

\item The polynomials $p(x)$ and $q(x)$ fulfill
\begin{equation}\label{equacionprincipal}
mp'q-npq'=\wt{\lambda}.
\end{equation}

\smallskip

\item The polynomials $p(x)$ and $q(x)$ fulfill
\begin{equation}\label{equacionpotencia}
p^m-q^n = n\lambda x^{mn-m-n+1}+\text{lower order terms,}
\end{equation}

\smallskip

\item Write
$$
p(x)= \prod_{i=1}^n (x-\alpha_i)\quad \text{and}\quad q(x)= \prod_{j=1}^m (x-\beta_j).
$$
The polynomial $g:=pq\in K[x]$ is separable and fulfills
\begin{equation}\label{equacionsymmetrica}
m g'(\alpha_i)=\wt{\lambda}\quad\text{and}\quad n g'(\beta_i)= -\wt{\lambda}.
\end{equation}

\smallskip

\end{enumerate}
\end{proposition}

\begin{proof} $(1)\Leftrightarrow (2)$.\enspace This follows directly using the evaluation
map at $Y = 1$ in one direction and taking $C_{-k}:=c_{-k}Y^{k+1}$ in the other direction.

\smallskip

\noindent $(2)\Rightarrow (5)$.\enspace We know that $p = c^n$ and there exists
$$
f = \lambda x^{1-n} + f_{-n}x^{-n}+ f_{-n-1}x^{-n-1}+\cdots\in K((x^{-1}))
$$
such that $c^m = q+f$. Hence
$$
p^m  = c^{mn} = (q+f)^n = q^n + nq^{n-1}f + \binom{n}{2} q^{n-2}f^2 +\cdots
$$
Since
$$
\deg(q^{n-k}f^k)=m(n-k)+k(1-n)\qquad\text{and}\qquad q^{n-1}f=\lambda x^{mn-m-n+1} + \text{lower order terms.}
$$
this implies item~(5).

\smallskip

\noindent $(5)\Rightarrow (2)$.\enspace An standard computation shows that there exists a unique
$$
c = x + c_0 + c_{-1}x^{-1}+ c_{-2}x^{-2} + \cdots \in K((x^{-1})),
$$
such that $c^n = p$. Write $f:=c^m-q$. Since the leading terms of $q$ and $c^m$ coincide,
$\deg(f)<m$. Furthermore
$$
c^{nm}= q^n+nfq^{n-1} + r,
$$
where $r\in K[x]$ has degree lower than $\deg(fq^{n-1})$. On the other hand, by hypothesis,
$$
c^{nm}-q^n = p^m-q^n = n\lambda x^{mn-m-n+1} + \text{lower order terms},
$$
and so,
$$
nfq^{n-1} +r = n\lambda x^{mn-m-n+1} + \text{lower order terms}.
$$
Since $q$ is monic of degree $m$, this implies that $\deg(f) = 1-n$ and the principal
coefficient $f_{1-n}$ of $f$ is $\lambda$.

\smallskip

\noindent $(5)\Rightarrow (4)$.\enspace Set $j:=mn-m-n$ and write
\begin{equation}
t:= p^m - q^n - n\lambda x^{j+1}.\label{eq22}
\end{equation}
By hypothesis $\deg(t)\le j$. Computing the derivative in~\eqref{eq22}, we obtain
$$
mp^{m-1}p' = nq^{n-1}q'+(j+1)n\lambda x^j + t'.
$$
Multiplying this equality by $q$, and dividing the result by $p^{m-1}$, we get
$$
mqp'=nq'\frac{q^n}{p^{m-1}}+(j+1)n\lambda x^j \frac{q}{p^{m-1}}+t'\frac{q}{p^{m-1}}.
$$
But, by~\eqref{eq22}
$$
\frac{q^{n}}{p^{m-1}}=p-\frac{n\lambda x^{j+1}}{p^{m-1}}-\frac{t}{p^{m-1}},
$$
and so
$$
mqp'=npq'-\frac{n^2\lambda x^{j+1}q'}{p^{m-1}} + (j+1)n\lambda x^j\frac{q}{p^{m-1}} -
\frac{ntq'}{p^{m-1}} + t'\frac{q}{p^{m-1}}.
$$
Since $p$ and $q$ are polynomials,
\begin{align*}
&\deg\left(\frac{n^2\lambda x^{j+1}q'}{p^{m-1}} \right) = 0
\text{ and its principal coefficient is $n^2m\lambda$,}\\
&\deg\left((j+1)n\lambda x^j\frac{q}{p^{m-1}}\right) = 0
\text{ and its principal coefficient is $(j+1)n\lambda$}
\shortintertext{and}
& \deg\left(\frac{ntq'}{p^{m-1}}\right),\,\deg\left(t'\frac{q}{p^{m-1}}\right) < 0,
\end{align*}
we conclude that
$$
mqp' = npq'+n\lambda(1-m-n),
$$
as desired.

\smallskip

\noindent $(4)\Rightarrow (5)$.\enspace By hypothesis
$$
\left(\frac{p^m}{q^n}\right)'=\frac{mp^{m-1}p'q^n-nq^{n-1}q'p^m}{q^{2n}}=
\frac{(mp'q-nq'p)q^{n-1}p^{m-1}}{q^{2n}}=\wt{\lambda}\frac{p^{m-1}}{q^{n+1}}.
$$
Since
$$
\deg\left(\wt{\lambda}\frac{p^{m-1}}{q^{n+1}}\right)=-m-n
\text{ and its principal coefficient is $\wt{\lambda}$,}
$$
there a exist $\kappa\in K$ and $r\in K((x^{-1}))$ such that $\deg(r) = 1-m-n$,
the principal coefficient of $r$ is $\wt{\lambda}/(1-m-n)$ and
$$
\frac{p^m}{q^n}=\kappa + r.
$$
Moreover, since $\deg(p^m) = \deg(q^n)$ and $p,q$ are monic, $\kappa = 1$. Hence,
$$
p^m=q^n+\frac{\wt{\lambda}}{1-n-m}\,x^{mn-m-n+1}+\text{terms of lower order,}
$$
as desired.

\smallskip

\noindent $(3)\Leftrightarrow (4)$.\enspace A direct computation shows that
\begin{align*}
[P,Q] & = P_x Q_Y-P_Y Q_x\\
& = \sum_{i=0}^n i p_i x^{i-1} Y^{n-i}\sum_{j=0}^m (m-j) q_j x^j Y^{m-j-1}\\
& - \sum_{i=0}^n (n-i) p_i x^i Y^{n-i-1}\sum_{j=0}^m j q_j x^{j-1} Y^{m-j}\\
& = \sum_{i,j} p_i q_j (i(m-j)-(n-i)j) x^{i+j-1} Y^{m+n-i-j-1}\\
&= \sum_{i,j} p_i q_j (mi-nj) x^{i+j-1} Y^{m+n-i-j-1}
\end{align*}
and
\begin{align*}
mp'q-npq' & = m \sum_{i=0}^n i p_i x^{i-1} \sum_{j=0}^m q_j x^j - n
\sum_{i=0}^n p_i x^i \sum_{j=0}^m j q_j x^{j-1}\\
& = \sum_{i,j} p_i q_j mi x^{i+j-1} - \sum_{i,j} p_i q_j nj x^{i+j-1}\\
&= \sum_{i,j} p_i q_j (mi-nj) x^{i+j-1}.
\end{align*}
So, it is clear that $[P,Q] = \wt{\lambda} Y^{m+n-2}$ if and only if $mp'q-npq'=\wt{\lambda}$.

\smallskip

\noindent $(4)\Rightarrow (6)$.\enspace A direct computation shows that
\begin{equation}
mg'-(m+n)pq'=mp'q-npq'=(m+n)qp'-ng'.\label{eq23}
\end{equation}
Using item~(4) and evaluating the first equality at the $\alpha_i$'s and the second
one at the $\beta_j$'s, we obtain~\eqref{equacionsymmetrica}. Since $\lambda\ne 0$,
this implies that $g$ has no multiple roots, and so $g$ is separable.

\smallskip

\noindent $(6)\Rightarrow (4)$.\enspace By equalities~\eqref{eq23} and the hypothesis, we have
$$
(mp'q-npq')(\alpha_i) = (mp'q-npq')(\beta_j) = \wt{\lambda}\quad\text{for all $i,j$.}
$$
Since $\deg(mp'q-npq') \le n+m-1$, this implies that $mp'q-npq' = \wt{\lambda}$.
\end{proof}

\begin{proposition}\label{proposicion1} Let $n,m>1$. If $n|m$ or $m|n$, then there is no
solution to \eqref{equacionprincipal}.
\end{proposition}

\begin{proof} Assume that $mp'q-npq' = \wt{\lambda}$ and $m=nk$ with $k\in \mathds{N}$.
Set $\ov{q}:=q-p^k$. Then
$$
\ov{q}(x)=a_r x^r+\text{lower degree terms}\qquad\text{for some $0\le r<m$.}
$$
On one hand
$$
mp'\ov{q}-np\ov{q}'=\wt{\lambda},
$$
but, on the other hand the leading term of $mp'\ov{q}-np\ov{q}'$ is $n a_r x^{n+r-1}(m-r)$.
Hence $n+r-1=0$, which contradicts $n>1$ and $r\ge 0$.
\end{proof}

\begin{proposition}\label{zannier}
If $m\nmid n$ and $n\nmid m$, then the system $S(n,m,\lambda)$ has at least one solution.
\end{proposition}

\begin{proof}
Set $\mu_1=\mu_2=\dots=\mu_n:=m$ and $\nu_1=\nu_2=\dots=\nu_m:=n$. Clearly
$$
mn=\sum_{i=1}^{m}\mu_i=\sum_{j=1}^{n}\nu_j.
$$
Moreover $\delta:=\gcd(m,n)<m,n$, which implies that
$$
\max\left\{mn\frac{\delta-1}{\delta},mn-m-n+1\right\}=mn-m-n+1.
$$
Hence, by~\cite{Z}*{Theorem 1, page 114} there exist polynomials $F$, $G$ having
$\mu_i$, resp. $\nu_j$  as the sequences of multiplicities
of their roots, satisfying
$$
\deg(F-G)=mn-m-n+1,
$$
and it is evident that we can assume that $F$ and $G$ are monic.
But then $F(x)=p(x)^m$, where $p(x)$ is the product of the linear factors of $F$ and similarly
$G(x)=q(x)^n$ with $q(x)$ monic. Then
$$
p(x)^m-q(x)^n=F-G= n\mu x^{mn-m-n+1}+lower\ order\ terms
$$
for some $\mu\in K^{\times}$.
Using the automorphism of $K[x]$ given by $x\mapsto x-p_{n-1}/n$ we achieve $p_{n-1}=0$.
Hence, the condition~(5) of Proposition~\ref{proposicion} is satisfied, and by that
proposition the pair $(p,q)$ is associated to a solution of $S(n,m,\mu)$.
Let $\alpha\in K^{\times}$ be such that $\alpha^{n+m-1}=\lambda/\mu$. Replacing $p_i$ by $\alpha^{n-i}p_i$ and
$q_i$ by $\alpha^{m-i}q_i$ for all $i$, we obtain a solution of $S(n,m,\lambda)$, as desired.
\end{proof}

By definitions two pairs $(p,q)$ and $(p_1,q_1)$ of monic polynomials in $K[x]$ are
$\infty$-equivalent if there are $a\in K^{\times}$ and $b\in K$ such that
$$
p_1(x) = a^{-\deg(p)} p(a x +b)\qquad\text{and}\qquad q_1(x) = a^{-\deg(q)} q(a x +b).
$$
Theorem~\ref{Jinvertible} and Propositions~\ref{proposicion} and~\ref{proposicion1} show that
$S_t(n,m,Y^{m+n-1})$ has finitely many solutions. This yields an alternative proof
of a result contained in Theorem~4 of~\cite{F}, which says that the
equation~\eqref{equacionprincipal} has only finitely many solutions for fixed $m,n$,
modulo $\infty$-equivalence. In fact we have:

\begin{proposition} Assume that $K$ is algebraically closed and let $m$, $n$ be
positive integers. Then there are only finitely many $\infty$-equivalence classes
of pairs of monic polynomials $p, q \in K[x]$ such that $p$ has degree $n$, $q$ has
degree $m$, and $mp'q-npq'$ is equal to $\wt{\lambda}$ for some $\wt{\lambda}\in K^{\times}$.
\end{proposition}

\begin{proof} Let $\mathcal{S}$ be the set of pairs $(p,q)$ of monic polynomials in
$K[x]$ of degree $n$ and $m$, respectively, such that
$$
mp'q-npq' = 1\qquad\text{and}\qquad p = x^n + p_{n-2} x^{n-2} + \cdots + p_0.
$$
By Theorem~\ref{Jinvertible} and Propositions~\ref{proposicion} and~\ref{proposicion1} we know
that $\mathcal{S}$ is a finite set. So in order to finish the proof it suffices to
show that if $(\tilde{p},\tilde{q})$ is a pair of monic polynomials in $K[x]$ of
degree $n$ and $m$ respectively such that
$$
m\tilde{p}'\tilde{q}-n\tilde{p}\tilde{q}'  = \wt{\lambda},
$$
where $\wt{\lambda}\in K^{\times}$, then $(\tilde{p},\tilde{q})$ is $\infty$-equivalent
to a pair a $(p,q)\in \mathcal{S}$. But for this it suffices to take
$$
p(x) := a^{-n} \tilde{p}(ax - \tilde{p}_{n-1}/n)\qquad\text{and}\qquad q(x) :=
a^{-n} \tilde{q}(ax - \tilde{p}_{n-1}/n),
$$
where $\tilde{p}_{n-1}$ is the coefficient of $x^{n-1}$ in $\tilde{p}$ and $a\in K$
satisfies $a^{m+n-1} = \wt{\lambda}$.
\end{proof}

Moreover, we have additional information about the set $\mathcal{S}_0$ of solutions
$(p,q)$ of~\eqref{equacionprincipal} satisfying that $p$ and $q$ are monic,
$\deg(p) = n$, $\deg(q) = m$ and the coefficient of $x^{n-1}$ in $p$ is zero. Let
$e:=m+n-1$ and assume that $K$ has a primitive $e$-root of unit. The group
$\mathds{Z}/e\mathds{Z}$ acts on $\mathcal{S}_0$. In fact, if
$(c_{-1},\dots,c_{-k},\dots,c_{-m-n+2})$ is a solution of $S(n,m,\lambda)$ in $K^{m+n-2}$,
then $(c_{-1}u^{2i},\dots,c_{-k}u^{(k+1)i},\dots, c_{-m-n+2}u^{(m+n-1)i})$ is also a
solution of $S(n,m,\lambda)$ in $K^{m+n-2}$, and so we can define
$$
i\cdot (c_{-1},\dots,c_{-k},\dots,c_{-m-n+2}) :=
(c_{-1}u^{2i},\dots,c_{-k}u^{(k+1)i},\dots, c_{-m-n+2}u^{(m+n-1)i}).
$$
One can also check that if $n=2$ and $m=2r+1$, then there are exactly $r+1$ solutions
(all in the same orbit). It is not clear in which cases there are orbits with $m+n-1$
elements. We pose the following questions:
\begin{enumerate}

\smallskip

\item Let $d$ be a divisor of $m+n-1$ and assume $\{ m\pmod d,n\pmod d\}=\{0,1\}$. Does there
exist always an orbit of solutions of $S(n,m,Y^{m+n-1})$ with $\frac{m+n-1}{d}$
elements, such that $C_{-k}=0$ for $k+1\not\equiv 0 \pmod d$?

\smallskip

\item Let $\phi$ be the Euler function. If $\phi(m+n-1)>2$, does there exists an
orbit in the solution set of $S(n,m,Y^{m+n-1})$ with $m+n-1$ elements?

\smallskip

\end{enumerate}

\smallskip

In~\cite{F} the author also considers the equation
\begin{equation}\label{otra de Formanek}
mp'q-npq'=\lambda p
\end{equation}
where $\lambda\in K^{\times}$. This equation is strongly related with equation~\eqref{equacionprincipal} by
the following:
$$
mp'q-npq'=\lambda\Longrightarrow (m+n)p'Q-npQ'=\lambda p,
$$
where $Q:=pq$.

\smallskip

For the rest of the section we will prove the following proposition, which answers partially question~(2) in a particular case:

\begin{proposition}\label{existencia ralo}
Let $d$ be a divisor of $m+n-1$ and let $r:=\gcd(m,n)$. Assume that $d>r$ and that $\{ m\pmod d,n\pmod d\}=\{0,1\}$.
Then there exists always a solution $C$ of $S(n,m,1)$ such that $C_{-k}=0$ for $k+1\not\equiv 0 \pmod d$.
\end{proposition}

 Let $A_1$ be the polynomial
$K$-algebra $K[Z_{-r-d},Z_{r-2d},Z_{r-3d},\dots]$ in the variables $Z_{r-vd}$, with $v>0$.
Consider the Laurent series
$$
Z := x^r+Z_{r-d}x^{r-d}+Z_{r-2d}x^{r-2d}+\cdots\in A_1((x^{-1})).
$$
Set $N:=(m+n-1)/d$ and assume, without loss of generality, that
$$
m=1\pmod d\qquad\text{and}\qquad n=0\pmod d.
$$
Let $\lambda\in K$ and let $\tilde C\in K((x^{-1}))$ be a solution of $S(n,m,\lambda)$ with $\tilde C_{-k}=0$ for $k+1\not\equiv 0 \pmod d$. If we define $C:=\tilde C^r$, then the coefficients $C_{r-d},\dots,C_{r-Nd}$ of $C$ satisfy the $N$ equations
\begin{equation}
\begin{aligned}
G_k& :=(Z^{n/r})_{-dk} = 0, && \text{for $k=1,\dots, (m-1)/d$,}\\
G_{k+(m-1)/d} & :=(Z^{m/r})_{-dk+1} = 0, &&\text{for $k=1,\dots,n/d-1$,}\\
G_N & :=(Z^{m/r})_{-n+1} + \lambda = 0.
\end{aligned}\label{sistema de ecuaciones modificado}
\end{equation}
(Note that $Z_{r-Nd}$ is the the lowest degree coefficient of $Z$ which appears
in the system. It appears in the equation $(Z^{n/r})_{m-1}=0$ and in the last equation).

\begin{lemma}\label{casocero} Let $d:=\gcd(n,m)$ and $j\in \mathds{N}_0$. Let $P\in x^j K[x]$ be a monic
polynomial of degree~$n$ and let
$$
C = x + C_0 x^0 +  C_{-1} x^{-1} +  C_{-2} x^{-2} + \cdots \in K((x^{-1}))
$$
be such that $C^n = P$. If $(C^m)_{-k}=0$ for $k=1,\dots,n-\max(j,1)$, then $C^d\in K[x]$.
\end{lemma}

\begin{proof} Write $C^m=Q+F$ where $Q\in K[x]$ and $F\in x^{-1}K[[x^{-1}]]$. Since $P\in x^j K[x]$, we have
\begin{equation}\label{eq1}
G := mP'Q-nQ'P\in \begin{cases} x^{j-1}K[x] &\text{if $j>0$,}\\ K[x] &\text{if $j=0$.}\end{cases}
\end{equation}
We claim that $G=0$. Since,
$$
G = m n C^{n-1}C'(C^m-F)-n C^n (m C^{m-1}C'-F')= n F' C^n - mnFC^{n-1}C'
$$
and, by hypothesis, $\deg(F)\le \max(j,1)-n-1$, if $G\ne 0$, then $\deg(G) \le \max(j,1)-2$, which is
impossible by equality~\eqref{eq1}. Thus the claim follows. But then
$$
\left(\frac{P^m}{Q^n}\right)'=\frac{mP^{m-1}P'Q^n-nQ^{n-1}Q' P^m}{Q^{2n}} =\frac{P^{m-1}Q^{n-1}}{Q^{2n}}(mP' Q-nQ'P)=0,
$$
which combined with the fact that $P$ and $Q$ are monic, implies that $Q^n = P^m$. Consequently there exists a monic
polynomial $R$ such that $P = R^{n/d}$, and so $C^d = R \in K[x]$, as desired.
\end{proof}

\begin{proposition}\label{le radical} Let $I$ be the ideal of $K[Z_{r-d},\dots, Z_{r-Nd}]$ generated by
$G_1,\dots,G_{N-1},G^{(0)}$, where $G^{(0)}:=(Z^{m/r})_{1-n}$. Then $\sqrt{I}=
\langle Z_{r-d},\dots, Z_{r-Nd} \rangle$.
\end{proposition}

\begin{proof} By the Nullstellensatz it suffices to prove that $V(I) = \{(0,\dots,0)\}$, where $V(I)$ denotes the
Zero-locus of the ideal $I$. So take a solution
$$
c:=(C_{r-d},\dots,C_{r-Nd})\in  K^{N}
$$
of $G_1,\dots,G_{N-1},G^{(0)}$, and set
$$
C := x^r+C_{r-d}x^{r-d}+C_{r-2d}x^{r-2d}+\cdots+C_{r-Nd}x^{r-Nd}\in x^rK[[x^{-d}]].
$$
Clearly
\begin{equation}\label{se anulan}
(C^{n/r})_{-k}=0\quad\text{for $k=1,\dots,m-1$}\qquad\text{and}\qquad
(C^{m/r})_{-k}=0\quad\text{for $k=1,\dots,n-1$.}
\end{equation}
Now, by a similar argument as in Remark~\ref{re 1.21}, there exists
$$
C_{r-Nd-d},C_{r-Nd-2d},C_{r-Nd-3d},\dots\in K,
$$
such that the
$$
\ov{C}:= x^r+ \sum_{k=1}^{\infty}  C_{r-kd}x^{r-kd}\in x^rK[[x^{-d}]]
$$
still satisfies~\eqref{se anulan} and such that the monic $r$-root of $C$,
$$
\tilde C:= x + \tilde C_{1-d}x^{1-d}+ \tilde C_{1-2d}x^{1-2d}+ \tilde C_{1-3d}x^{1-3d}+\cdots \in x K[[x^{-d}]]
$$
is a solution of $S(n,m,0)$. Hence $P:=\tilde C^n$ is a monic polynomial of degree~$n$ and we can apply
Lemma~\ref{casocero} with $j=0$. Hence $\ov{C} = \tilde C^r \in K[x]$ and so, $C_{r-d}=0,\dots,C_{r-Nd}=0$ because $d>r$.
This means that $c = (0,\dots,0)$, as desired.
\end{proof}

\begin{corollary}\label{no pertenece ralo} Let $I_1$ be the ideal of $K[Z_{r-d},\dots, Z_{r-Nd}]$ generated
by $G_1,\dots,G_{N-1}$. Then $G^{(0)}\notin \sqrt{I_1}$.
\end{corollary}

\begin{proof} If we assume that $G^{(0)}\in \sqrt{I_1}$, then by Proposition~\ref{le radical} we have $\sqrt{I_1} =
\langle Z_{r-d},\dots, Z_{r-Nd}\rangle$, which is impossible since $I_1$ is generated by $N-1$ elements and
the height of $\langle Z_{r-d},\dots, Z_{r-Nd}\rangle$ is $N$.
\end{proof}

\begin{proof}[Proof of Proposition~\ref{existencia ralo}] By Corollary~\ref{no pertenece ralo} and the Nullstellensatz,
there exists
$$
C=(C_{r-d},\dots,C_{r-Nd})\in K^N
$$
such that $G_i(C) = 0$ for $1\le i< N$, but $G^{(0)}(C)\ne 0$. Let
$$
\tilde C:= x + \tilde C_{1-d}x^{1-d}+ \tilde C_{1-2d}x^{1-2d}+ \tilde C_{1-3d}x^{1-3d}+\cdots \in x K[[x^{-d}]]
$$
be the monic $r$-root in $x K[[x^{-d}]]$ of the Laurent series $\ov{C}$ determined by $C$ as in Remark~\ref{re 1.21}. Then
$$
(\tilde C_{-1},\dots, \tilde C_{-Nd+1})
$$
is a solution of $S_t(n,m,\lambda)$, where $\lambda:=-G^{(0)}(C)$. Let $\alpha\in K$ be such that $\alpha^{N}=1/\lambda$
and set $\hat C_{1-id}:=\alpha^{i} \tilde C_{1-id}$. It is clear that $\hat C:=(\hat C_{-1},\dots,\hat C_{-Nd+1})$ is a solution
of $S_t(n,m,1)$. As in Remark~\ref{re 1.21}, this determines a solution
$\check{C}$ of $S(n,m,1)$. It is easy to check that $\check{C}_{-k}=0$ for $k+1\not\equiv 0 \pmod d$, as desired.
\end{proof}

\section{A modified system and an example}
In this section we modify the system~\eqref{sistema de ecuaciones} in order to verify
one of the 4 exceptional cases found by Moh in~\cite{M}. The case $(m,n)=(48,64)$ has
been already be verified independently in~\cite{H} and~\cite{GGV2}. We will verify the
case $(m,n)=(50,75)$. Doing this directly using~\eqref{sistema de ecuaciones} amounts
to solving a system of $123$ equations and 123 variables. Due to this we take  an
alternative strategy. The first part of this procedure is similar to the one used
in~\cite{GGV1}*{Section~8}, and is inspired by~\cite{M}. We do not provide proofs for
this first part, since it serves only to verify a known case and to show the usefulness
of systems like~\eqref{sistema de ecuaciones}. Let $A_0$ and $\gamma$ be as in the
discussion above \cite{GGV1}*{Proposition~6.2}. Assume there is a counterexample
$(P_0,Q_0)$ to the Jacobian conjecture with $\deg(P_0)=50$ and $\deg(Q_0)=75$.
Then by~\cite{GGV1}*{Remark~7.10}, we know that $A_0=(5,20)$. Futhermore, using
similar computations as in \cite{GGV1}*{Proposition~8.3}, one can check that
necessarily $\gamma=3$ or $\gamma=2$. Proceeding as in~\cite{GGV1}*{Section~8} we
obtain a pair $(P_1,Q_1)\in K[x,y]$, such that
$$
[P_1,Q_1]=x^2,\quad \deg(P_1)=10\quad\text{and}\quad \deg(Q_1)=15.
$$
If $\gamma=3$, then applying to $(P_1,Q_1)$ first the automorphism $x \mapsto xy^3$,
$y \mapsto y^{-2}$ of $K[x,y,y^{-1}]$, and then the automorphism $x\mapsto x-G$, $y\mapsto y$
for some suitable $G\in K[y,y^{-1}]$, we obtain a pair $(P,Q)\in K[x,y,y^{-1}]$ satisfying:

\begin{enumerate}

\smallskip

\item[(a1)] There exist $\lambda\in K$, $\mu\in K^{\times}$ and $C,F\in K[y,y^{-1}]((x^{-1}))$ such that
$$
P=C^2\qquad\text{and}\qquad Q=C^3+\lambda C^{-1}+F,
$$

\smallskip

\item[(a2)] $[P,Q]=\mu  y^6 (x-G)^2$, for some $G\in K[y,y^{-1}]$,

\smallskip

\item[(a3)] there exists $f_2,f_4,f_6,f_8\in K$ such that
$$
F=F_{-1} x^{-1}+F_{-2} x^{-2}+F_{-3} x^{-3}+\cdots,
$$
with $F_{-1}:=y^7$ and $F_{-2} := f_8 y^8 + f_6 y^6 + f_4 y^4 + f_2 y^2$,

\smallskip

\item[(a4)] $C=x^2+C_0+C_{-1}x^{-1}+\cdots$,

\smallskip

\item[(a5)] $\deg_y(C_{-k})\le k+2$ for all $k\ge 0$,

\smallskip

\item[(a6)] $C_0=c_{0,2} y^2+c_{0,0}+c_{0,-2} y^{-2}+\cdots+ c_{0,-10} y^{-10}$, with $c_{0,-10}\ne 0$.

\smallskip

\end{enumerate}
On the other hand, if $\gamma=2$, then applying to $(P_1,Q_1)$ first the automorphism
$x \mapsto xy^2$, $y \mapsto y^{-3}$ of $K[x,y,y^{-1}]$, and then the automorphism
$x\mapsto x-G$, $y\mapsto y$ for some suitable $G\in K[y,y^{-1}]$, we obtain a pair
$(P,Q)\in K[x,y,y^{-1}]$ satisfying:

\begin{enumerate}

\smallskip

\item[(b1)] There exist $\lambda\in K$, $\mu\in K^{\times}$ and $C,F\in K[y,y^{-1}]((x^{-1}))$ such that
$$
P=C^2\qquad\text{and}\qquad Q=C^3+\lambda C^{-1}+F,
$$

\smallskip

\item[(b2)] $[P,Q]=\mu  y^2 (x-G)^2$, where $G:= g_{-2} y^{-2} + g_{-5} y^{-5}$, with $g_{-2},g_{-5}\in K$,

\smallskip

\item[(b3)] there exist  $f_2,f_{-1},f_{-4},f_{-7},b_1,b_{-2}\in K$ such that
$$
F=F_{-3} x^{-3}+F_{-4} x^{-4}+F_{-5} x^{-5}+\cdots,
$$
with $F_{-3}:=y^3$, $F_{-4} := b_1 y + b_{-2} y^{-2}$ and $F_{-5} :=
f_2 y^2 + f_{-1} y^{-1} + f_{-4} y^{-4} + f_{-7} y^{-7}$,

\smallskip

\item[(b4)] $C=x^3+C_1 x+C_0+C_{-1}x^{-1}+\cdots$,

\smallskip

\item[(b5)] $C_{-1}=c_{-1,1} y+ c_{-1,-2} y^{-2}+\cdots+ c_{-1,-17}y^{-17}+ c_{-1,-20}y^{-20}$,
with $c_{-1,1}\ne 0$,

\smallskip

\item[(b6)] $C_1 = e_{-1} y^{-1} + e_{-4} y^{-4} + e_{-7} y^{-7} + e_{-10} y^{-10}$ and
$e_{-10}\ne 0$ if $C_0= 0$.

\smallskip

\end{enumerate}
We first analyze the case $\gamma=3$. Motivated by (a4), we consider the Laurent series
$$
Z := x^2+Z_0+Z_{-1}x^{-1}+Z_{-2}x^{-2}+\cdots\in K[Z_0,Z_{-1},Z_{-2},\dots]((x^{-1})).
$$
We set
\begin{equation}
\begin{aligned}
E_k:=(Z^2)_{-k} & , && \text{for $k=1,\dots, 5$,}\\
E_{5+k}:=\left( Z^3+\lambda Z^{-1}\right)_{-k} & , &&\text{for $k=1,\dots,3$.}
\end{aligned}
\end{equation}
Explicitly, we have
\begin{align*}
E_1 =& 2 Z_0 Z_{-1} + 2 Z_{-3},\\
E_2 =& Z_{-1}^2 + 2 Z_0 Z_{-2} + 2 Z_{-4},\\
E_3 =& 2 Z_{-1} Z_{-2} + 2 Z_0 Z_{-3} + 2 Z_{-5},\\
E_4 =& Z_{-2}^2 + 2 Z_{-1} Z_{-3} + 2 Z_0 Z_{-4} + 2 Z_{-6},\\
E_5 =& 2 Z_{-2} Z_{-3} + 2 Z_{-1} Z_{-4} + 2 Z_0 Z_{-5} + 2 Z_{-7},\\
E_6 =& 3 Z_0^2 Z_{-1} + 6 Z_{-1} Z_{-2} + 6 Z_0 Z_{-3} + 3 Z_{-5},\\
E_7 =& \lambda+3 Z_0 Z_{-1}^2 + 3 Z_0^2 Z_{-2} + 3 Z_{-2}^2 + 6 Z_{-1} Z_{-3} + 6 Z_0 Z_{-4} +3 Z_{-6}, \\
E_8 =& Z_{-1}^3 + 6 Z_0 Z_{-1} Z_{-2} + 3 Z_0^2 Z_{-3} + 6 Z_{-2} Z_{-3} + 6 Z_{-1} Z_{-4} + 6 Z_0 Z_{-5} + 3 Z_{-7}.
\end{align*}
Note that $Z_{-7}$ is the lowest degree coefficient of $Z$ which appears in the $E_i$'s.
It appears in the term $2 Z_{-7}$ of $E_5$ and in the term $3 Z_{-7}$ of  $E_{8}$. If
$C\!\in\! K[y,y^{-1}]((x^{-1}))$ fulfills (a1)--(a6), then the~8 coefficients
$C_1,C_0,C_{-1},\dots,C_{-7}$, of $C$, satisfy the equations
\begin{equation*}
E_1=\dots= E_5=0,\quad E_{6} =-F_{-1},\quad E_{7}=-F_{-2}\quad\text{and}\quad E_{8}=-F_{-3}.
\end{equation*}
From $E_1=0$, $E_3=0$ and $E_6=-F_{-1}$ we obtain $F_{-1}+3 C_{-1}C_{-2}=0$. Setting
$$
F_{-1}:=-3 C_{-1}C_{-2}
$$
and eliminating in the set of equations
$$
E_2=\dots=E_5 = 0,\quad E_6 = -F_{-1}\quad\text{and}\quad E_7 = -F_{-2},
$$
the variables $C_{-3}$, $C_{-4}$, $C_{-5}$, $C_{-6}$ and $C_{-7}$, we obtain
$$
C_0 (3 C_0 C_{-1}^2 - 3 C_{-2}^2 - 2 \lambda) = 2 C_0 F_{-2}.
$$
But using that  $y^7+3 C_{-1}C_{-2}=0$ and that by (a5) we have $\deg_y(C_{-1})\le 3$ and
$\deg_y(C_{-2})\le 4$, we get $C_{-1}=a y^3$ and $C_{-2}=b y^4$ for some $a,b\in K^{\times}$.
Hence, either $C_0= 0$ or
$$
C_0=\frac{3 C_{-2}^2 +2 F_{-2}+ 2 \lambda}{3 C_{-1}^2}=\frac{2 \lambda}{3 a^2 y^6} + \frac{2 f_2}{3 a^2 y^4}
+ \frac{2 f_4}{3 a^2 y^2} + \frac{2 f_6}{3 a^2}  + \frac{b^2 y^2}{a^2} + \frac{2 f_8 y^2}{3 a^2},
$$
which contradicts that by (a6) we have $c_{0,-10}\ne 0$. This rules out the case~$\gamma=3$.

\smallskip

We now analyze the case $\gamma=2$. Motivated by (b4) we consider the Laurent series
$$
Z := x^3+Z_1 x+Z_0+Z_{-1}x^{-1}+Z_{-2}x^{-2}+\cdots\in K[Z_1,Z_0,Z_{-1},Z_{-2},\dots]((x^{-1})).
$$
We set
\begin{equation}
\begin{aligned}
E_k:=(Z^2)_{-k} & , && \text{for $k=1,\dots, 8$,}\\
E_{8+k}:=\left( Z^3+\lambda Z^{-1}\right)_{-k} & , &&\text{for $k=1,\dots,5$.}
\end{aligned}
\end{equation}
Explicitly we have
\allowdisplaybreaks
\begin{align*}
E_1 =& 2 Z_ 0 Z_ {-1} + 2 Z_ 1 Z_ {-2} + 2 Z_ {-4},\\
E_2 =& (Z_ {-1})^2 + 2 Z_ 0 Z_ {-2} + 2 Z_ 1 Z_ {-3} + 2 Z_ {-5},\\
E_3 =& 2 Z_ {-1} Z_ {-2} + 2 Z_ 0 Z_ {-3} + 2 Z_ 1 Z_ {-4} + 2 Z_ {-6},\\
E_4 =& (Z_ {-2})^2 + 2 Z_ {-1} Z_ {-3} + 2 Z_ 0 Z_ {-4} +  2 Z_ 1 Z_ {-5} + 2 Z_ {-7},\\
E_5 =& 2 Z_ {-2} Z_ {-3} + 2 Z_ {-1} Z_ {-4} + 2 Z_ 0 Z_ {-5} + 2 Z_ 1 Z_ {-6} + 2 Z_ {-8},\\
E_6 =& (Z_ {-3})^2 + 2 Z_ {-2} Z_ {-4} + 2 Z_ {-1} Z_ {-5} +  2 Z_ 0 Z_ {-6} + 2 Z_ 1 Z_ {-7} + 2 Z_ {-9},\\
E_7 =&  2 Z_ {-3} Z_ {-4} + 2 Z_ {-2} Z_ {-5} +  2 Z_ {-1} Z_ {-6} + 2 Z_ 0 Z_ {-7}
+ 2 Z_ 1 Z_ {-8}+2 Z_ {-10},\\
E_8 =& (Z_ {-4})^2 + 2 Z_ {-3} Z_ {-5} + 2 Z_ {-2} Z_ {-6} + 2 Z_ {-1} Z_ {-7} + 2 Z_ 0 Z_ {-8}
+ 2 Z_ 1 Z_ {-9}+ 2 Z_ {-11},\\
E_ 9 =& 3 (Z_ 0)^2 Z_ {-1} + 3 Z_ 1 (Z_ {-1})^2 +  6 Z_ 0 Z_ 1 Z_ {-2} + 3 (Z_ {-2})^2 + 3 (Z_ 1)^2 Z_ {-3}
+ 6 Z_{-1} Z_ {-3}\\
& + 6 Z_ 0 Z_ {-4} + 6 Z_ 1 Z_ {-5} + 3 Z_ {-7},\\
E_ {10} =& 3 Z_ 0 (Z_ {-1})^2 + 3 (Z_ 0)^2 Z_ {-2} + 6 Z_ 1 Z_ {-1} Z_ {-2} + 6 Z_ 0 Z_ 1 Z_ {-3}
+ 6 Z_ {-2} Z_ {-3}+ 3 (Z_ 1)^2 Z_ {-4}\\
  &  + 6 Z_ {-1} Z_ {-4} + 6 Z_ 0 Z_ {-5} + 6 Z_ 1 Z_ {-6} + 3 Z_ {-8},\\
E_ {11} =& \lambda + (Z_ {-1})^3 + 6 Z_ 0 Z_ {-1} Z_ {-2} + 3 Z_ 1 (Z_ {-2})^2 + 3 (Z_ 0)^2 Z_ {-3}
+ 6 Z_ 1 Z_ {-1} Z_ {-3} + 3 (Z_ {-3})^2 \\
& + 6 Z_ 0 Z_ 1 Z_ {-4} + 6 Z_ {-2} Z_ {-4} +  3 (Z_ 1)^2 Z_ {-5} + 6 Z_ {-1} Z_ {-5} + 6 Z_ 0 Z_ {-6}
 +6 Z_ 1 Z_ {-7} + 3 Z_ {-9},\\
E_ {12} =& 3 (Z_ {-1})^2 Z_ {-2} + 3 Z_ 0 (Z_ {-2})^2 + 6 Z_ 0 Z_ {-1} Z_ {-3}
+ 6 Z_ 1 Z_ {-2} Z_ {-3} + 3 (Z_ 0)^2 Z_ {-4} + 6 Z_ 1 Z_ {-1} Z_ {-4}\\
& + 6 Z_ {-3} Z_ {-4} +6 Z_ 0 Z_ 1 Z_ {-5}  + 6 Z_ {-2} Z_ {-5}
+ 3 (Z_ 1)^2 Z_ {-6} + 6 Z_ {-1} Z_ {-6} + 6 Z_ 0 Z_ {-7}\\
& +3 Z_ {-10}+ 6 Z_ 1 Z_ {-8},\\
E_ {13} =& -\lambda Z_ 1  + 3 Z_ {-1} (Z_ {-2})^2 + 3 (Z_ {-1})^2 Z_ {-3}
+ 6 Z_ 0 Z_ {-2} Z_ {-3} + 3 Z_ 1 (Z_ {-3})^2 + 6 Z_ 0 Z_ {-1} Z_ {-4}\\
  & + 6 Z_ 1 Z_ {-2} Z_ {-4} + 3 (Z_ {-4})^2
  + 3 (Z_ 0)^2 Z_ {-5} + 6 Z_ 1 Z_ {-1} Z_ {-5} + 6 Z_ {-3} Z_ {-5} + 6 Z_ 0 Z_ 1 Z_ {-6} \\
  & + 6 Z_ {-2} Z_ {-6} + 3 (Z_ 1)^2 Z_ {-7}
  + 6 Z_ {-1} Z_ {-7} + 6 Z_ 0 Z_ {-8} + 6 Z_ 1 Z_ {-9}+ 3 Z_ {-11}.
\end{align*}
Note that $Z_{-11}$ is the lowest degree coefficient of $Z$ which appears in the $E_i$'s.
It appears in the term $2 Z_{-11}$ of $E_8$ and in the term $3 Z_{-11}$ of  $E_{13}$.
If $C\!\in\! K[y,y^{-1}]((x^{-1}))$ fulfills (b1)--(b6), then the~13 coefficients
$C_1,C_0,C_{-1},\dots,C_{-11}$ of $C$, satisfy the equations
\begin{equation}\label{ecuaciones para el caso 13}
E_1=\dots=E_{10} = 0,\quad E_{11}=-y^3,\quad E_{12}=-F_{-4}\quad\text{and}\quad E_{13}=-F_{-5}.
\end{equation}
First we will prove that $F_{-4}=0$. Assume $F_{-4}\ne 0$. Eliminating in the set of equations
$$
E_1=\dots=E_7 = 0,\quad E_9,E_{10} = 0\quad\text{and}\quad E_{12} =-F_{-4}
$$
the variables $C_{0}$, $C_{1}$, $C_{-3}$, $C_{-5}$, $C_{-6}$, $C_{-7}$, $C_{-8}$ and $C_{-9}$, we obtain
$$
C_{-1}^2 C_{-4} = C_{-2}^3\quad\text{and}\quad 2 F_{-4} = 3 C_{-1}^2 C_{-2}.
$$
Since $F_{-4}= b_1 y + b_{-2} y^{-2}$ and $C_{-1}\in yK[y^{-3}]$ by (b5), necessarily
$C_{-1}$ is homogeneous, and so $C_{-1}=c_{-1,1} y$. For the sake of simplicity we
write $a:=c_{-1,1}$. We set $F_{-3}=y^3$,  $C_{-4} := C_{-2}^3/C_{-1}^2$  and
$C_{-2} := 2 F_{-4}/3 C_{-1}^2$, and in the set of equations
$$
E_1=\dots=E_7 = 0,\quad E_9 = 0,\quad E_{10} = 0,\quad E_{11}=-F_{-3}\quad\text{and}\quad E_{12}=-F_{-4}
$$
we eliminate the variables $C_{1}$, $C_{-3}$, $C_{-5}$, $C_{-6}$, $C_{-7}$, $C_{-8}$,
$C_{-9}$ and $C_{-10}$. This yields
$$
864 F_{-4}^2 \lambda = -\frac{256 F_{-4}^6}{a^{10} y^{10}} +
\frac{864 C_0 F_{-4}^3}{a y} - 864 F_{-4}^2 y^3 + 432 a^3 F_{-4}^2 y^3 - 729 a^8 C_0^2 y^8,
$$
from which we deduce
$$
(27 a^9 C_0 y^9 - 16 F_{-4}^3)^2=432 a^{10} F_{-4}^2 y^{10} (-2 \lambda + (-2 + a^3) y^3).
$$
This implies that $-2 \lambda + (-2 + a^3) y^3$ is a square in $K((y^{-1}))$, which is only possible if
\begin{equation}\label{para a3}
a^3=2.
\end{equation}
Now we compute
$$
[P,Q]=[P,F]=[x^6,F_{-3} x^{-3}]+[x^6,F_{-4} x^{-4}]+[x^6,F_{-5} x^{-5}]+[2 C_1 x^4,F_{-3} x^{-3}].
$$
Using this, (b2) and the expressions for $F_{-3}$, $F_{-4}$, $F_{-5}$ $C_1$ and
$G$ given in (b2), (b3) and (b6), we obtain
$$
6 b_1 + 36 g_{-2} - \frac{12 b_{-2}}{y^3} + \frac{36 g_{-5}}{y^3}=0
$$
and
$$
-\frac{18 g_{-5}^2}{y^8} - \frac{36 e_{-10}}{y^8} - \frac{42 f_{-7}}{y^8} -
\frac{36 g_{-2} g_{-5}}{y^5} - \frac{ 18 e_{-7}}{y^5} - \frac{24 f_{-4}}{y^5} -
\frac{18 g_{-2}^2}{y^2} - \frac{6 f_{-1}}{y^2} + 18 e_{-1} y + 12 f_2 y=0.
$$
Hence
\begin{align*}
&f_2 = -\frac{3 e_{-1}}{ 2},\quad f_{-1} = -3 g_{-2}^2,\quad  f_{-4} =
-\frac 34 (2 g_{-2} g_{-5} + e_{-7}),\quad f_{-7} = -\frac 37 (g_{-5}^2 + 2 e_{-10}),\\
&b_1 = -6 g_{-2},\quad b_{-2} = 3 g_{-5}.
\end{align*}
Now eliminating from the system~\eqref{ecuaciones para el caso 13} all variables except $C_{-1}$, we obtain
\begin{align*}
R_0 :=& C_{-1}^{10} (3 C_{-1}^9 - 36 C_{-1}^2 F_{-5}^2 + 18 C_{-1}^6 F_{-3} -
96 F_{-3}^3 - 6 C_{-1}^6 \lambda - 48 C_{-1}^3 F_{-3} \lambda -
      96 F_{-3}^2 \lambda)\\
& - (C_{-1}^6 F_{-5} F_{-4}^2 (-48 C_{-1}^3 - 96 F_{-3}) + F_{-4}^4 (16 C_{-1}^6
+ 64 C_{-1}^3 F_{-3} + 64 F_{-3}^2))=0,
\end{align*}
and eliminating from the same system all variables except $C_{-1}$ and $C_{1}$, we obtain among others
$$
R_1 := 4 F_{-4}^2 - C_{-1}^3 (3 C_1 C_{-1}^3 + 12 F_{-5} + 12 C_1 F_{-3})=0.
$$
Equating to zero the coefficients of $R_0$ and $R_1$, and taking into
account~\eqref{para a3}, we obtain the system of equations:
\allowdisplaybreaks
\begin{align*}
0=& a^3-2\\
0=&-\frac 37 (-12 (7 + a^3) g_{-5}^2 + a^3 (4 + 7 a^3) e_{-10}),\\ 0 =&
-3 (-6 (-8 + a^3) g_{-2} g_{-5} + a^3 (1 + a^3) e_{-7}),\\ 0 =& -3 (4 + a^3) (-12 g_{-2}^2 + a^3 e_{-4}),\\
0 =& -3 a^3 (-2 + a^3) e_{-1},\\
0=& -\frac{324}{49} ((28 + 14 a^3 + a^6) g_{-5}^2 + 2 a^6 e_{-10})^2,\\
0 =& -\frac{162}{7} (2 (-32 - 16 a^3 + a^6) g_{-2} g_{-5} + a^6 e_{-7})
((28 + 14 a^3 + a^6) g_{-5}^2 + 2 a^6 e_{-10}), \\
0 =& -\frac{81}{28}(28 a^6 (-32 - 16 a^3 + a^6) g_{-2} g_{-5} e_{-7} + 7 a^12 e_{-7}^2\\
& +4 g_{-2}^2 (3 (3584 + 3584 a^3 + 864 a^6 - 16 a^9 + 5 a^12) g_{-5}^2 + 16 a^6 (4 + a^3)^2 e_{-10})),\\
0 =& -\frac{162}{7} (14 (4 + a^3)^2 (-32 - 16 a^3 + a^6) g_{-2}^3 g_{-5} + 7 a^6 (4 + a^3)^2 g_{-2}^2 e_{-7} \\
&+2 a^6 e_{-1} ((28 + 14 a^3 + a^6) g_{-5}^2 + 2 a^6 e_{-10})),\\
0 =& -81 (4 (4 + a^3)^4 g_{-2}^4 + 2 a^6 (-32 - 16 a^3 + a^6) g_{-2} g_{-5} e_{-1}
+ a^{12} e_{-1} e_{-7}),\\ 0 =& -324 a^6 (4 + a^3)^2 g_{-2}^2 e_{-1},\\
0 =& -3 a^{10} (27 a^2 e_{-1}^2 + 32 \lambda + 16 a^3 \lambda + 2 a^6 \lambda),\\
0 =& 3 a^{10} (-32 + 6 a^6 + a^9).
\end{align*}
Eliminating in this system the variables $a$, $e_{-10}$, $e_{-7}$, $e_{-4}$, $e_{-1}$
and $\lambda$, we obtain $g_{-2}^5 = 0$ and $g_{-5}^4 = 0$. So, $F_{-4}=
\frac{3 g_{-5}}{y^2} - 6 g_{-2} y=0$, as desired.

Now, eliminating from the set of equations $E_1=\dots=E_{10}=0$, $E_{12}=0$ all variables
except $C_0$ and $C_{-1}$, we obtain $C_0 C_{-1}^4 = 0$  (hence $C_0=0$), and eliminating
from the set of equations $E_1=\dots=E_{10}=0$, $E_{11}+F_{-3}=0$ and $E_{12}=0$ all
variables except $C_1$ and $C_{-1}$, we obtain among others
$$
8 C_{-1}^2 F_{-3} = C_{-1}^2 (-3 C_1^2 C_{-1}^2 + 4 C_{-1}^3 - 8 \lambda),
$$
which implies that
$$
C_{-1}^2(4C_{-1} - 3C_1^2) = 8(F_{-3} + \lambda)=8(y^3 + \lambda),
$$
because $C_{-1}\ne 0$. Hence $C^{-1}$ is homogeneous, since it belongs to $yK[y^{-3}]$. Write $C_{-1}=ay$. Then
$$
3 a^2 C_1^2 y^2 = -8 \lambda - 8 y^3 + 4 a^3 y^3.
$$
But the right hand side can be only a square in $K((y^{-1}))$ if $a^3=2$, and then $C_1$
is homogeneous with $\deg_y(C_1)=-1$, i.e. $e_{-10}=0$, which contradicts (b6), since
$C_0=0$. This rules out the case~$\gamma=2$.

\begin{remark}
  The formulas in~\cite{HLLN}*{Theorem 1.1} could help in order to obtain explicitly system of equations for $C_i$.
  In the language of~\cite{HLLN} our $C$ is $F^{1/d}$, and our equations come from the case $\mu>e$.
\end{remark}

\begin{remark}
While searching for a counterexample to
the Jacobian Conjecture with $\frac{\deg(P)}{\deg(Q)}=\frac 23$, we often encounter a pair $(P,Q)$ such that
  there exist $\lambda\in K$ and $C,F\in K[y,y^{-1}]((x^{-1}))$ with
$$
P=C^2\qquad\text{and}\qquad Q=C^3+\lambda C^{-1}+F.
$$
(See for example conditions (a1) and (b1) above.)

In particular we have $[P,Q]\in K^{\times}$, $\deg_x(P)=2k$, $\deg_x(Q)=3k$ and
\begin{equation}\label{cubo menos cuadrado}
  \deg_x(P^3-Q^2-2\lambda P)=\deg_x(FC^3)=k+1.
\end{equation}
Note that, since $[P,Q]=P_x Q_y-P_y Q_x\in K^{\times}$, we know that $P'=P_x$ and $Q'=Q_x$ are coprime polynomials in $R[x]=(K[y])[x]$.
Hence, it could be interesting to characterize coprime polynomials $P,Q\in R[x]$ with $P',Q'$ also coprime,
such that $\deg(P^3-Q^2-\lambda P)$ is minimal. For $\lambda=0$ we don't need the condition on $P'$, $Q'$, and we recover the notion of
Davenport-Zannier pairs as in~\cite{APZ}. If $\lambda\ne 0$, the condition on $P'$ and $Q'$ is necessary, as the following example shows:
$$
P = x^{2k} + 2c, \  Q = x^{3k} + 3cx^k,
$$
and then $P^3 - Q^2 - 3c^2P = 2c^3\in R^{\times}$ has degree zero (This example was communicated by Leonid Makar-Limanov).

We could call the resulting pairs non-homogeneous Davenport-Zannier pairs, and for $k=2$ we have the following example with $R=K[y]$:
$$
P=y^2 x^4-2y x^3+x^2+y x-\frac 12\quad\text{and}\quad
Q=y^3 x^6-3y^2 x^5+3y x^4+\left( \frac{3y^2}{2}-1\right)x^3-\frac{9y}{4}x^2+\frac 34 x+\frac{3y}{8}.
$$
Then $[P,Q]=\frac{3y}{8}$ and so $P'=P_x$ and $Q'=Q_x$ are coprime. Moreover,
$$
P^3-Q^2-\frac{3}{16} P=\frac{y^3}{8}x^3+\frac{3y^2}{16}x^2-\frac{1}{64}(2+9y^2)
$$
has degree $3=k+1$.

We finish this article by asking the following questions
\begin{enumerate}
  \item If coprime polynomials $P,Q\in R[x]$ have degree $2k$ and $3k$ respectively,
  and $P'$ and $Q'$ are also coprime, what is the minimal degree of $P^3-Q^2-\lambda P$ for $\lambda\ne 0$? In particular, is this minimal degree equal to $k+1$? (Note that
  we can assume $\lambda=1$).
  \item Are there always pairs with $\deg(P^3-Q^2-\lambda P)=k+1$?
  \item Is it possible to characterize all pairs as above, which we call non homogeneous
  Davenport-Zannier pairs, similarly to the characterization in~\cite{APZ}?    
\end{enumerate}
\end{remark}

\paragraph{Acknowledgment}
We wish to thank Leonid Makar-Limanov for pointing out the result of~\cite{Z}, and for the example above.


\begin{bibdiv}
\begin{biblist}

\bib{A}{book}{
   author={Abhyankar, S. S.},
   title={Lectures on expansion techniques in algebraic geometry},
   series={Tata Institute of Fundamental Research Lectures on Mathematics
   and Physics},
   volume={57},
   note={Notes by Balwant Singh},
   publisher={Tata Institute of Fundamental Research},
   place={Bombay},
   date={1977},
   pages={iv+168},
   review={\MR{542446 (80m:14016)}},
}

\bib{APZ}{book}{
   author={Adrianov, Nikolai M.},
   author={Pakovich, Fedor},
   author={Zvonkin, Alexander K.},
   title={Davenport-Zannier polynomials and dessins d'enfants},
   series={Mathematical Surveys and Monographs},
   volume={249},
   publisher={American Mathematical Society, Providence, RI},
   date={2020},
   pages={xi+187},
   isbn={978-1-4704-5634-4},
   review={\MR{4249449}},
}

\bib{F}{article}{
   author={Formanek, Edward},
   title={Theorems of W. W. Stothers and the Jacobian conjecture in two
   variables},
   journal={Proc. Amer. Math. Soc.},
   volume={139},
   date={2011},
   number={4},
   pages={1137--1140},
   issn={0002-9939},
   review={\MR{2748408 (2011m:14104)}},
   doi={10.1090/S0002-9939-2010-10523-3},
}

\bib{GGV1}{article}{
author={Guccione, Jorge Alberto},
author={Guccione, Juan Jos\'e},
author={Valqui, Christian},
   title={On the shape of possible counterexamples to the Jacobian
   Conjecture},
   journal={J. Algebra},
   volume={471},
   date={2017},
   pages={13--74},
   issn={0021-8693},
}

\bib{GGV2}{article}{
author={Guccione, Jorge Alberto},
author={Guccione, Juan Jos\'e},
author={Valqui, Christian},
   title={A Differential Equation for Polynomials Related to the
Jacobian Conjecture},
   journal={Pro-Mathematica},
   volume={27},
   date={2013},
   pages={83-98},
   issn={1012-3938},
}

\bib{H}{article}{
   author={Heitmann, R},
   title={On the Jacobian conjecture},
   journal={Journal of Pure and Applied Algebra},
   volume={64},
   date={1990},
   pages={35--72},
   issn={0022-4049},
   review={\MR{1055020 (91c :14018)}},
}

\bib{HLLN}{article}{
   author={Hurst, William E.},
   author={Lee, Kyungyong},
   author={Li, Li},
   author={Nasr, George D.},
   title={On the two-dimensional Jacobian conjecture: Magnus' formula
   revisited, I},
   journal={Rocky Mountain J. Math.},
   volume={53},
   date={2023},
   number={3},
   pages={791--806},
   issn={0035-7596},
   review={\MR{4617912}},
   doi={10.1216/rmj.2023.53.791},
}

\bib{K}{article}{
   author={Keller, Ott-Heinrich},
   title={Ganze Cremona-Transformationen},
   language={German},
   journal={Monatsh. Math. Phys.},
   volume={47},
   date={1939},
   number={1},
   pages={299--306},
   issn={0026-9255},
   review={\MR{1550818}},
   doi={10.1007/BF01695502},
}

\bib{M}{article}{
   author={Moh, T. T.},
   title={On the Jacobian conjecture and the configurations of roots},
   journal={J. Reine Angew. Math.},
   volume={340},
   date={1983},
   pages={140--212},
   issn={0075-4102},
   review={\MR{691964 (84m:14018)}},
}

\bib{vdE}{book}{
   author={van den Essen, Arno},
   title={Polynomial automorphisms and the Jacobian conjecture},
   series={Progress in Mathematics},
   volume={190},
   publisher={Birkh\"auser Verlag},
   place={Basel},
   date={2000},
   pages={xviii+329},
   isbn={3-7643-6350-9},
   review={\MR{1790619 (2001j:14082)}},
   doi={10.1007/978-3-0348-8440-2},
}

\bib{Z}{article}{
   author={Zannier, Umberto},
   title={On Davenport's bound for the degree of $f^3-g^2$ and
   Riemann's existence theorem},
   journal={Acta Arith.},
   volume={71},
   date={1995},
   number={2},
   pages={107--137},
   issn={0065-1036},
   review={\MR{1339121 (96k:11029a)}},
}

\end{biblist}
\end{bibdiv}
\end{document}